\documentclass[11pt, twoside]{article}

\usepackage{amssymb}
\usepackage{amsmath}
\usepackage{amsthm}
\usepackage{color}
\usepackage{mathrsfs}

\allowdisplaybreaks

\def\rr{{\mathbb R}}
\def\rn{{{\rr}^n}}
\def\zz{{\mathbb Z}}

\def\nn{{\mathbb N}}

\def\cg{{\mathcal G}}

\def\cq{{\mathcal Q}}

\def\cs{{\mathcal S}}

\def\fz{\infty}
\def\az{\alpha}
\def\bz{\beta}
\def\dz{\delta}

\def\ez{\epsilon}

\def\lz{\lambda}

\def\oz{{\omega}}

\def\vz{\varphi}

\def\lf{\left}
\def\r{\right}

\def\hs{\hspace{0.25cm}}
\def\ls{\lesssim}
\def\gs{\gtrsim}
\def\tr{\triangle}

\def\noz{\nonumber}
\def\wz{\widetilde}

\def\ev{:=}
\def\st{\subset}

\def\gfz{\genfrac{}{}{0pt}{}}

\def\supp{\mathop\mathrm{\,supp\,}}

\def\lp{{L^p(\rn)}}
\def\dt{\,\frac{dt}t}

\newcommand{\dbt}{{\dot{B}_{p,q}^{s,\tau}(\rn)}}
\newcommand{\dft}{{\dot{F}_{p,q}^{s,\tau}(\rn)}}
\newcommand{\ft}{{F_{p,q}^{s,\tau}(\rn)}}

\newcommand{\dlt}{{\dot{L}_{p,q,a}^{s,\tau}(\mathcal{G})}}
\newcommand{\dpt}{{\dot{P}_{p,q,a}^{s,\tau}(\mathcal{G})}}

\newcommand{\dlh}{{L\dot{H}_{p,q,a}^{s,\tau}(\mathcal{G})}}
\newcommand{\dph}{{P\dot{H}_{p,q,a}^{s,\tau}(\mathcal{G})}}

\newcommand{\dbh}{{B\dot{H}_{p,q}^{s,\tau}(\rn)}}
\newcommand{\dfh}{{F\dot{H}_{p,q}^{s,\tau}(\rn)}}

\newcommand{\dbhr}{B\dot{H}_{p,q}^{s,\tau}(\rn)}
\newcommand{\dfhr}{F\dot{H}_{p,q}^{s,\tau}(\rn)}

\def\bmo{{{\mathop\mathrm {BMO}\,(\rn)}}}

\newcommand{\N}{{\mathbb N}}
\newcommand{\R}{{\mathbb R}}

\newcommand{\Z}{{\mathbb Z}}

\newcommand{\CoY}{\ensuremath{\mbox{Co}Y}}
\newcommand{\Co}{\ensuremath{\mbox{Co}}}

\newtheorem{theorem}{Theorem}[section]
\newtheorem{lemma}[theorem]{Lemma}
\newtheorem{corollary}[theorem]{Corollary}
\newtheorem{proposition}[theorem]{Proposition}
\theoremstyle{definition}
\newtheorem{remark}[theorem]{Remark}
\newtheorem{definition}[theorem]{Definition}

\numberwithin{equation}{section}

\begin{document}

\arraycolsep=1pt

\title{\bf\Large New Characterizations of
Besov-Triebel-Lizorkin-Hausdorff Spaces Including
Coorbits and Wavelets\footnotetext{\hspace{-0.35cm}
This work is partially supported by the GCOE program
of Kyoto University and 2010 Joint Research Project Between
China Scholarship Council and German Academic Exchange
Service (PPP) (Grant No. LiuJinOu [2010]6066). Yoshihiro Sawano is supported
by Grant-in-Aid for Young Scientists (B) (Grant No. 21740104) of
Japan Society for the Promotion of Science.
Dachun Yang is supported by the National
Natural Science Foundation (Grant No. 11171027) of China
and Program for Changjiang Scholars and Innovative
Research Team in University of China. Wen Yuan is supported by the National
Natural Science Foundation (Grant No. 11101038) of China.
}}
\author{Yiyu Liang, Yoshihiro Sawano, Tino Ullrich,
Dachun Yang\,\footnote{Corresponding author}\ \ and Wen Yuan}
\maketitle

\noindent{\bf Abstract.} In this paper, the authors establish new
characterizations of the recently introduced Besov-type spaces
$\dot{B}^{s,\tau}_{p,q}({\mathbb R}^n)$ and Triebel-Lizorkin-type
spaces $\dot{F}^{s,\tau}_{p,q}({\mathbb R}^n)$ with $p\in
(0,\infty]$, $s\in{\mathbb R}$, $\tau\in [0,\infty)$, and $q\in
(0,\infty]$, as well as their preduals, the Besov-Hausdorff spaces
$B\dot{H}^{s,\tau}_{p,q}(\R^n)$ and Triebel-Lizorkin-Hausdorff
spaces $F\dot{H}^{s,\tau}_{p,q}(\R^n)$, in terms of the local means,
the Peetre maximal function of local means, and the tent space
(the Lusin area function) in both discrete and continuous types.
As applications, the authors then obtain
interpretations as coorbits in the sense of H. Rauhut in
[Studia Math. 180 (2007), 237-253] and discretizations via
the biorthogonal wavelet bases for the full range of parameters
of these function spaces. Even for some special cases of this setting
such as $\dot F^s_{\infty,q}({\mathbb R}^n)$ for $s\in{\mathbb R}$, $q\in
(0,\infty]$ (including $\mathop\mathrm{BMO}
({\mathbb R}^n)$ when $s=0$ and $q=2$), the $Q$ space
$Q_\alpha ({\mathbb R}^n)$,
the Hardy-Hausdorff space $HH_{-\alpha}({\mathbb R}^n)$ for
$\alpha\in (0,\min\{n/2,1\})$, the Morrey space
${\mathcal M}^u_p({\mathbb R}^n)$ for $1<p\le u<\infty$,
and the Triebel-Lizorkin-Morrey space
$\dot{\mathcal{E}}^s_{upq}({\mathbb R}^n)$
for $0<p\le u<\infty$, $s\in{\mathbb R}$
and $q\in(0,\infty]$,
some of these results are new.

\bigskip


\bigskip


\vspace{0.4cm}

\section{Introduction}\label{s1}

In the last few decades, function spaces have been one of the
central topics in modern harmonic analysis, and are widely used in
various areas such as the potential theory, partial differential
equations, and the approximation theory. The Besov-type space
$\dbt$, the Triebel-Lizorkin-type space $\dft$, and their preduals,
the Besov-Hausdorff space $\dbh$ and Triebel-Lizorkin-Hausdorff
space $\dfh$ were recently introduced and investigated in
\cite{yy1,yy2,syy,ysay,ysy}. These spaces establish the connection
between the classical Besov-Triebel-Lizorkin spaces (see, for
example, \cite{t83,fj2,t92}), the $Q$ spaces in \cite{ejpx}, and the
Hardy-Hausdorff spaces in \cite{dx}, which has been posed as an open
question in \cite{dx}. It was shown in \cite{yy1,yy2,syy} that these
spaces unify and generalize many classical function spaces including
classical Besov and Triebel-Lizorkin spaces, Triebel-Lizorkin-Morrey
spaces (see, for example, \cite{tx, st, s09, syy}), $Q$ spaces, and
Hardy-Hausdorff spaces (see also \cite{X1,X2}).

These spaces are usually defined through building blocks constructed
out of a dyadic decomposition of unity on the Fourier side. Several
of the mentioned applications make it necessary to use more general
convolution kernels for the definition of the spaces, especially the
so-called local means of a function are of particular interest. By
the pioneering work of Bui, Paluszy\'nski and Taibleson
\cite{BuPaTa96, BuPaTa97}, and later by Rychkov \cite{Ry99a, R99,
R01}, we know that one rather uses kernels satisfying the Tauberian
conditions (see \eqref{3.2} below), for characterizing classical
Besov-Triebel-Lizorkin spaces. This also applies to the setting of
the spaces $\dbt$, $\dft$, $\dbh$ and $\dfh$ considered here, which
represents one of the main results of the paper; see Theorems
\ref{t3.1}, \ref{t3.2}, \ref{t3.3} and \ref{t3.4} below. To establish these results we
use a key estimate (see Lemma \ref{l3.2} below) which was obtained
in \cite{u10} by a variant of a method from Rychkov \cite{R99,R01}
and originally from Str\"omberg and Torchinsky
\cite[Chapter 5]{st89}. In addition we make use of a certain involved
decomposition of $\rn$ into proper sub-cubes in combination with
some localized modifications of the approaches used for Besov and
Triebel-Lizorkin spaces. For the technically more difficult
Hausdorff type spaces $\dbh$ and $\dfh$, we need to incorporate
the geometrical properties of the Hausdorff capacities (see \cite{ysay}
or Lemma \ref{l3.3} below) and make use of the well-known
Aoki-Rolewicz theorem in \cite{ao,Ro57}.
Indeed, theses spaces are known to be quasi-Banach spaces; see
\cite{ysay, ysy}. In \cite{yy3}, equivalent (quasi-)norms of the
above spaces were already established via the discrete local means and
the discrete Peetre maximal function of local means, which closely
follows the idea of Triebel \cite{t88} (see also \cite{t92}).
However, these proofs rely on the fact that the respective function
is contained in the space under consideration. Thus, these
(quasi-)norms can not yet be considered as characterizations of the
considered space.

Motivated by Ullrich \cite{u10}, in this paper, we extend the
discrete characterizations via the local means and the Peetre maximal
function by some further, technically convenient, related
characterizations and its continuous counterparts. In particular, we
establish a characterization via tent spaces (the Lusin area function).
Recall that the tent space was originally invented by Coifman, Meyer
and Stein in \cite{cms} and has nowadays been proved to be a useful
tool in harmonic analysis and partial differential equations.

Despite the fact that Theorems \ref{t3.1}, \ref{t3.2}, \ref{t3.3} and \ref{t3.4} are of
independent interest for the theory of these spaces, our main reason
for establishing these continuous characterizations is the coorbit
space theory for quasi-Banach spaces developed by Rauhut
\cite{ra05-3} (see also Section \ref{Classcoo} below), which is a
continuation of the classical coorbit space theory developed by
Feichtinger and Gr\"ochenig \cite{FeGr86, FeGr89a, FeGr89b, Gr88,
Gr91}. Similar to \cite{u10}, but with the main advantage that we
also incorporate (quasi-)Banach spaces, we interpret the spaces
$\dbt$, $\dft$, $\dbh$ and $\dfh$ as coorbits of certain spaces on
the $ax+b$-group based on Theorems \ref{t3.1}, \ref{t3.2}, \ref{t3.3} and \ref{t3.4}
from Section \ref{s3} of this paper. Here, the main difficulty lies
in the fact that, in the (quasi-)Banach situation, a coorbit space
is defined as a certain retract of a Wiener-Amalgam space; see
\cite{ra05-3}. That is why we introduce four classes of Peetre-type
spaces, $\dlt$, $\dpt$, $\dlh$ and $\dph$, on the $n$-dimensional
$ax+b$-group $\cg$ (see Definitions \ref{d5.1} and \ref{d5.2} below)
in order to recover the spaces $\dbt$, $\dft$, $\dbh$ and $\dfh$.
As a first step we prove that the left and right translations of
$\cg$ are bounded on these Peetre-type spaces; see Propositions
\ref{p5.1} and \ref{p5.2} below. Especially for the Hausdorff-type
spaces this strongly depends on the geometrical properties of
Hausdorff capacities; see Lemmas \ref{5.1} and \ref{l5.2} below.
Let us emphasize once more that we particularly extend the results
in \cite{u10} to the quasi-Banach case, namely, to $\min\{p,q\}<1$.

By combining the interpretation as coorbit spaces in Section
\ref{ax+b} with the abstract discretization results in \cite{ra05-3}
(see also Section \ref{Classcoo} below), we finally obtain the
characterizations by \emph{biorthogonal wavelets} in Section
\ref{s6} for all admissible parameters of these spaces; see Theorems
\ref{t6.1} and \ref{t6.2} below. The main goal is to find
sufficient conditions for admissible wavelets. This reduces to the
task of translating abstract conditions \eqref{4.4} on the used
atoms in Subsection \ref{groupdisc} to this specific setting on the
$ax+b$-group, which results in a Wiener-Amalgam condition analyzed
in Proposition \ref{propwiener}. This condition can be ensured by
controlling the decay of the continuous wavelet transform; see Lemma
\ref{help1} below. Comparing with the wavelet characterizations of
$\dbt$ and $\dft$ obtained in \cite[Theorem 8.3]{ysy}, therein the
parameter $s\in(0,\infty)$, in Theorem \ref{t6.1}, the wavelet
characterizations of these spaces are established for all $s\in\rr$,
while the wavelet characterizations for Besov-Hausdorff spaces
$\dbh$ and Triebel-Lizorkin-Hausdorff spaces $\dfh$ in Theorem
\ref{t6.2} are totally new.

We point out that even for some special cases of $\dbt$, $\dft$,
$\dbh$ and $\dfh$, such as $\dot F^s_{\infty,q}({\mathbb R}^n)$ for
$s\in{\mathbb R}$ and $q\in (0,\infty]$ (including
$\mathop\mathrm{BMO} ({\mathbb R}^n)$ when $s=0$ and $q=2$; see
\cite{jn}), the $Q$ space $Q_\alpha ({\mathbb R}^n)$, the
Hardy-Hausdorff space $HH_{-\alpha}({\mathbb R}^n)$ for $\alpha\in
(0,\min\{n/2,1\})$ (\cite{dx,ejpx,X1,X2}), the Morrey space
${\mathcal M}^u_p({\mathbb R}^n)$ for $1<p\le u<\infty$
(\cite{m38}), and the Triebel-Lizorkin-Morrey space
$\dot{\mathcal{E}}^s_{upq}({\mathbb R}^n)$ for $0<p\le u<\infty$,
$s\in\rr$ and $q\in(0,\infty]$ (\cite{tx, st, s09, syy}), some
results of this paper are new.

We should mention that the Triebel-Lizorkin spaces $\dot{F}^s_{\fz,q}(\rn)$ when $q\in(1,\infty]$
were first considered by Triebel \cite{t83} in 1983. In 1990, Frazier and Jawerth \cite{fj2}
found a more appropriate definition of the spaces
$\dot F^s_{\infty,q}(\rn)$ via certain Carleson measure characterizations,
which had the advantage that it also works for all $q\in(0,\infty].$
Moreover, recently, it was proved in \cite{yy4} that, for all $s\in\rr$,
both the spaces $\dbt$ and $\dft$ coincide with
$\dot{B}^{s+n(\tau-1/p)}_{\infty,\infty}(\rn)$ if $\tau\in(1/p,\fz)$ and $q\in(0,\fz)$,
or $\tau\in[1/p,\infty)$ and $q=\fz$.
In particular, when $s+n(\tau-1/p)>0,$
the spaces $\dbt$ and $\dft$
coincide with H\"older-Zygmund spaces. Thus, Theorems \ref{t3.1}, \ref{t3.2}
and \ref{t6.1} also give new descriptions of these simple spaces, which could be
of interest on their own. Recall that the wavelet characterization of $\bmo$
is known for a long time, which was obtained by
Meyer in \cite[p.\,154, Theorem 4]{M92}.
We give, in Remark \ref{r6.-1} below, a detailed
comparison between \cite[p.\,154, Theorem 4]{M92} and the wavelet
characterization of $\bmo$ obtained as a special case of Theorem \ref{t6.1}(i)
by taking $s=0$, $p\in (0,\fz)$, $\tau=1/p$ and $q=2$.

We finally remark that some of results of this paper may also hold
for inhomogeneous variants of $\dbt$, $\dft$, $\dbh$ and $\dfh$.
Thanks to the recent work on generalized coorbit
space theory for quasi-Banach spaces \cite{sch} this can be done
with a similar method. However, to limit the length of this paper, we will not
go further to the details.

The paper is structured as follows. In Section \ref{s2} we recall
some necessary notation and some basic function spaces and their
properties. Section \ref{s3} is devoted to the spaces $\dbt$,
$\dft$, $\dbh$ and $\dfh$. We give their definitions and equivalent
characterizations via the local means. Section \ref{Classcoo} deals with
the main features from abstract coorbit space theory for
quasi-Banach spaces by recalling the main results from
\cite{ra05-3}. In Section \ref{ax+b}, we apply this abstract setting
to Peetre-type spaces on the $ax+b$-group and recover the spaces in
Section \ref{s3} as coorbits of Peetre-type spaces. Finally, in
Section \ref{s6}, we obtain characterizations with (bi)orthogonal
wavelets based on the results in Section \ref{ax+b}.

\section{Preliminaries\label{s2}}

In Subsection \ref{s2.1}, we recall some necessary notation
and, in Subsection \ref{s2.2}, some basic spaces of functions
and their properties which are used throughout the whole paper.

\subsection{Notation\label{s2.1}}

For all multi-indices $\az:=(\az_1,\cdots,\az_n)
\in (\nn\cup\{0\})^n$, let $\|\az\|_{\ell^1}:=|\az_1|+\cdots+|\az_n|$.
Let $\cs(\rn)$ be the {\it space of all Schwartz functions} on $\rn$
with the classical topology and $\cs'(\rn)$ its \emph{topological
dual space}, namely, the set of all
continuous linear functionals on $\cs(\rn)$ endowed with the weak
$\ast$-topology. Following Triebel \cite{t83}, we set
$$\cs_\infty(\rn):=\lf\{\varphi\in\cs(\rn):\ \int_\rn
\varphi(x)x^\gamma\,dx=0\ \mbox{for all multi-indices}\ \gamma\in
\lf(\nn\cup\{0\}\r)^n\r\}$$ and consider $\cs_\fz(\rn)$ as a
subspace of $\cs(\rn)$, including the topology. Use
$\cs'_\infty(\rn)$ to denote the {\it topological dual space} of
$\cs_\infty(\rn)$, namely, the set of all continuous linear
functionals on $\cs_\fz(\rn)$. We also endow $\cs'_\infty(\rn)$ with
the weak $\ast$-topology. Let $\mathcal{P}(\rn)$ be the {\it set of
all polynomials} on $\rn$. It is well known that
$\cs'_\fz(\rn)=\cs'(\rn)/\mathcal{P}(\rn)$ as topological spaces;
see, for example, \cite[Proposition 8.1]{ysy}. Similarly, for any
$N\in \nn\cup\{0\}$, the \emph{space} $\cs_N(\rn)$ is defined to be
the set of all Schwartz functions satisfying that $\int_\rn
\varphi(x)x^\gamma\,dx=0$ for all multi-indices $\|\gamma\|_{\ell^1}\le N$ and
$\cs'_N(\rn)$ its \emph{topological dual space}. We also let
$\cs_{-1}(\rn):=\cs(\rn)$. For any $\varphi \in \cs(\R^n)$, we
use $\widehat{\vz}$ or $\mathcal{F}\vz$ to denote its \emph{Fourier
transform}, namely, for all $\xi\in\rn$,
$\widehat{\vz}(\xi):=\mathcal{F}\vz(\xi) :=\int_\rn e^{-i\xi
x}\vz(x)\,dx$ and let $\varphi_j(x):= 2^{jn}\varphi(2^jx)$ for
all $j\in\mathbb{Z}$ and $x\in\rn$. Denote by $\mathcal{F}^{-1}\vz$
the \emph{inverse Fourier transform} of $\vz$. Let
$\widetilde{\vz}(x):=\overline{\vz(-x)}$ for all $x\in\rn$.

Throughout the whole paper, for all $\vz\in\cs(\rn)$ and distributions $f$ such
that $\vz\ast f$
makes sense, we let, for all $t\in(0,\infty)$, $k\in\zz$, $a\in(0,\fz)$
and $x\in\rn$,
$$(\vz_t^*f)_a(x):=\sup_{y\in\rn}\frac{|\vz_t\ast f(x+y)|}{(1+|y|/t)^a}
\ \text{and}\ (\vz_k^*f)_a(x):=\sup_{y\in\rn}\frac{|\vz_k\ast f(x+y)|}{(1+2^k|y|)^a},$$
which are called the \emph{Peetre-type maximal functions}.
In view of the above notation, we see that
$(\vz_k^*f)_a(x)=(\vz_{2^{-k}}^*f)_a(x)$. Since this difference
is always made clear in the context, we do not take care of this
abuse of notation.

For all $a,\,b\in\rr$, let $a\vee b:= \max\{a,\,b\}$ and
$a\wedge b:= \min\{a,\,b\}$.  For $j\in\zz$ and $k\in\zz^n$,
denote by $Q_{jk}$ the \textit{dyadic cube} $2^{-j}([0,1)^n+k)$,
$\ell(Q_{jk})$ its \textit{side length}, $x_{Q_{jk}}$ its \textit{lower
left-corner $2^{-j}k$} and $c_{Q_{jk}}$ its \textit{center}. Let
$\mathcal{Q}:= \{Q_{jk}:\ j\in\zz,\ k\in\zz^n\}$,
$\mathcal{Q}_j:=\{Q\in \mathcal{Q}:\ \ell(Q)=2^{-j}\}$
and $j_Q:=-\log_2 \ell(Q)$ for all $Q\in\mathcal{Q}$. When the
dyadic cube $Q$ appears as an index, such as
$\sum_{Q\in\mathcal{Q}}$ and $\{\cdot\}_{Q\in\mathcal{Q}}$, it is
understood that $Q$ runs over all dyadic cubes in $\rn$.

Throughout the whole paper, we denote by $C$ a \emph{positive
constant} which is independent of the main parameters, but it may
vary from line to line, while $C(\az, \bz, \cdots)$ denotes a
\emph{positive constant} depending on the parameters $\az$, $\bz$,
$\cdots$. The \emph{symbol} $A\ls B$ means that $A\le CB$. If $A\ls
B$ and $B\ls A$, then we write $A\sim B$. If $E$ is a subset of
$\rn$, we denote by $\chi_E$ the \emph{characteristic function} of
$E$. For all dyadic cubes $Q \in \mathcal{Q}$ and $r>0$, let $rQ$ be
the \emph{cube concentric with} $Q$ having the side length
$r\ell(Q)$. We also let $\nn:=\{1,\, 2,\, \cdots\}$ and
$\zz_+:=\nn\cup\{0\}$. For any $a\in\rr$, $\lfloor a\rfloor$
denotes the \emph{maximal integer} not larger than $a$.

\subsection{Basic Spaces of Functions\label{s2.2}}

In this subsection, we recall the notions of some basic
spaces of functions and their properties.

For all $p\in(0,\infty]$, the \emph{space} $L^p(\rn)$ is defined to
be the set of all complex-valued measurable functions $f$
(quasi-)normed by
$\|f\|_{\lp}:=\{\int_\rn |f(x)|^p\,dx\}^{1/p}$
with the usual modification when $p=\infty$.

Let $q\in(0,\fz]$ and $\tau\in[0,\fz)$. The \emph{space}
$\ell^q(L^p_\tau(\rn,\zz))$ with $p\in(0,\fz]$ is defined to be the
set of all sequences $G:=\{g_j\}_{j\in\zz}$ of measurable
functions on $\rn$ such that
$$\|G\|_{\ell^q(L^p_\tau(\rn,\zz))}
:= \sup_{P\in\mathcal{Q}}\frac1{|P|^\tau}
\lf\{\sum_{j=j_P}^\fz \lf[\int_P
|g_j(x)|^p\,dx\r]^{q/p}\r\}^{1/q}<\fz.$$
Similarly, the \emph{space} $L^p_\tau(\ell^q(\rn,\zz))$
with $p\in(0, \fz)$ is defined to be the
set of all sequences $G:=\{g_j\}_{j\in\zz}$ of
measurable functions on $\rn$ such that
$$\|G\|_{L^p_\tau(\ell^q(\rn,\zz))}
:= \sup_{P\in\mathcal{Q}}\frac1{|P|^\tau}
\lf\{\int_P \lf[\sum_{j=j_P}^\fz
|g_j(x)|^q\r]^{p/q}\,dx\r\}^{1/p}<\fz.$$

The following conclusions were obtained in \cite[Lemma 2.3]{yy3}.

\begin{lemma}\label{l2.1}
Let $q\in(0,\fz]$, $\tau\in[0,\fz)$ and $\dz\in(n\tau,\fz)$. Suppose
that $\{g_m\}_{m\in\zz}$ is a family of measurable functions on $\rn$. For
all $j\in\zz_+$ and $x\in\rn$, let
$G_j(x)\ev\sum_{m\in\zz}^\fz2^{-|m-j|\dz}g_m(x).$

{\rm (i)} If $p\in(0,\fz)$, then there exists a positive constant $C$,
independent of $\{g_m\}_{m\in\zz}$, such that
$\|\{G_j\}_{j\in\zz}\|_{\ell^q(L^p_\tau(\rn,\zz))}\le C
\|\{g_m\}_{m\in\zz}\|_{\ell^q(L^p_\tau(\rn,\zz))}.$

{\rm (ii)} If $p\in(0,\fz]$, then there exists a positive constant $C$,
independent of $\{g_m\}_{m\in\zz}$, such that
$\|\{G_j\}_{j\in\zz}\|_{L^p_\tau(\ell^q(\rn,\zz))}\le C
\|\{g_m\}_{m\in\zz}\|_{L^p_\tau(\ell^q(\rn,\zz))}.$
\end{lemma}

Next we recall the Hausdorff-type counterparts of
$\ell^q(L^p_\tau(\rn,\zz))$ and $L^p_\tau(\ell^q(\rn,\zz))$. To this
end, for $x\in\rn$ and $r>0$, let $B(x,\,r):= \{y\in\rn:\
|x-y|<r\}$. For $E \subset \rn$ and $d\in(0,\,n]$, the
\textit{$d$-dimensional Hausdorff capacity} of $E$ is defined by
\begin{equation}\label{2.1}
H^d(E):= \inf \lf\{\sum_jr_j^d:\ E\subset
\bigcup_jB(x_j,\,r_j)\r\},
\end{equation}
where the infimum is taken over all countable coverings
$\{B(x_j,\,r_j)\}_{j=1}^\infty$ of open balls of $E$; see, for
example, \cite{Ad, yy0}. It is well known that $H^d$ is monotone,
countably subadditive and vanishes at the empty set. Moreover, $H^d$
in \eqref{2.1} when $d=0$ also makes sense, and $H^0$ has the
properties that for all non-empty sets $E\subset\rn$, $H^0(E)\ge1$, and
$H^0(E)=1$ if and only if $E$ is bounded and non-empty.
For any function $f: \rn
\to [0,\fz]$, the \textit{Choquet integral} of $f$ with respect to
$H^d$ is then defined by
$$\int_{\rn}f(x)\,d H^d(x):= \int_0^{\fz}
H^d(\{x\in\rn:\ f(x)>\lz\})\,d\lz.$$

In what follows, we write $\R_+^{n+1}:= \rn\times (0,\fz)$. For all $p\in[1,\fz]$,
$p'$ denote the \emph{conjugate number of $p$}, namely, $1/p'+1/p=1$.
For any measurable function $\omega$ on $\R_+^{n+1}$ and $x\in\rn$,
its \textit{nontangential maximal function} $N\omega$ is defined by setting
$$N\omega(x):= \sup_{|y-x|<t} |\,\omega(y,t)|.$$

For $p\in(1,\fz)$ and $\tau\in[0,\fz)$, the \emph{space}
$\widetilde{\ell^q(L^p_\tau(\rn,\zz))}$ with $q\in[1,\fz)$
is defined to be the set of all sequences $G:=\{g_j\}_{j\in\zz}$ of measurable
functions on $\rn$ such that
$$\|G\|_{\widetilde{\ell^q(L^p_\tau(\rn,\zz))}}:= \inf_{\omega}
\lf\{\sum_{j\in\zz} \lf(\int_\rn |g_j(x)|^p[\omega(x,2^{-j})]^{-p}
\,dx\r)^{q/p}\r\}^{1/q}<\fz,$$
and the \emph{space} $\widetilde{L^p_\tau(\ell^q(\rn,\zz))}$
with $q\in(1,\fz)$ the set
of all sequences $G:=\{g_j\}_{j\in\zz}$ of measurable functions
on $\rn$ such that
$$\|G\|_{\widetilde{L^p_\tau(\ell^q(\rn,\zz))}}:= \inf_{\omega}
\lf\{\int_\rn \lf(\sum_{j\in\zz} |g_j(x)|^q[\omega(x,2^{-j})]^{-q}
\r)^{p/q}\,dx\r\}^{1/p}<\fz,$$
where the both infimums are taken over all
nonnegative Borel measurable functions $\omega$ on $\R_+^{n+1}$ satisfying
\begin{equation}\label{2.2}
\int_{\R^n} [N\omega(x)]^{(p \vee q)'} \,dH^{n\tau(p\vee q)'}(x) \le1
\end{equation}
and with the restriction that for any $j\in\zz$, $\omega(\cdot,
2^{-j})$ is allowed to vanish only where $g_j$ vanishes.

As an analogy of Lemma \ref{l2.1}, we have the following conclusions,
which is just \cite[Lemma 3.1]{yy3}.

\begin{lemma}\label{l2.2}
Let $p\in(1,\fz)$, $\delta\in(0,\fz)$ and $\{g_m\}_{m\in\zz}$ be a
sequence of measurable functions on $\rn$. For all
$j\in\zz$ and $x\in\rn$, let $G_j(x):= \sum_{m\in\zz}
2^{-|m-j|\delta} g_m(x)$.

{\rm (i)} If $q\in[1,\fz)$, $\tau\in[0,1/{(p\vee q)'}]$ and
$\delta\in(n\tau,\fz)$, then there exists a positive constant $C$,
independent of $\{g_m\}_{m\in\zz}$, such that
$\|\{G_j\}_{j\in\zz}\|_{\widetilde{\ell^q(L^p_\tau(\rn,\zz))}}\le C
\|\{g_m\}_{m\in\zz}\|_{\widetilde{\ell^q(L^p_\tau(\rn,\zz))}}.$

{\rm (ii)} If  $q\in(1,\fz)$, $\tau\in[0,1/{(p\vee q)'}]$ and
$\delta\in(n\tau,\fz)$, then there exists a positive constant $C$,
independent of $\{g_m\}_{m\in\zz}$, such that
$\|\{G_j\}_{j\in\zz}\|_{\widetilde{L^p_\tau(\ell^q(\rn,\zz))}}\le C
\|\{g_m\}_{m\in\zz}\|_{\widetilde{L^p_\tau(\ell^q(\rn,\zz))}}.$
\end{lemma}

For any locally integrable function $f$ on $\rn$ and $x\in\rn$, the
\textit{Hardy-Littlewood maximal function} $Mf(x)$ of $f$ is defined
by
$Mf(x):= \sup_{Q\ni x}\frac1{|Q|}\int_Q |f(y)|\,dy,$
where the supremum is taken over all cubes in $\rn$ centered at $x$ with sides
parallel to the coordinate axes. It is well known that $M$ is
bounded from $L^p(\rn)$ to $L^p(\rn)$ when $p\in(1,\infty]$; see,
for example, \cite{st2}. Moreover, if  $p\in(1,\infty)$ and $q\in(1,
\infty]$, then there exists a positive constant $C$ such that for
all sequences $\{f_k\}_{k\in\zz}$ of locally integrable functions on
$\rn$,
\begin{equation}\label{2.3}
\lf\|\lf\{\sum_{k\in\zz}[Mf_k]^q\r\}^{1/q}\r\|_{\lp}
\le
C \lf\|\lf\{\sum_{k\in\zz}|f_k|^q\r\}^{1/q}\r\|_{\lp}.
\end{equation}
This is the well-known Fefferman-Stein vector-valued inequality; see \cite{fs}
or \cite[p.\,56, (13)]{st2}.

\section{Continuous Characterizations}\label{s3}

In this section, we use the methods from \cite{u10} to characterize
the homogeneous Besov-type space $\dbt$, the Triebel-Lizorkin-type
space $\dft$ and their preduals, the Besov-Hausdorff space $\dbhr$
and the Triebel-Lizorkin-Hausdorff space $\dfhr$ via the local means,
the Peetre maximal functions of local means and the tent space
associated with local means (Lusin functions) in both discrete and
continuous types. These characterizations are further used in Section
\ref{ax+b} to prove that the spaces $\dbt$, $\dft$, $\dbhr$ and
$\dfhr$ are coorbits of certain spaces on the $n$-dimensional
$ax+b$-group $\mathcal{G}$.

\subsection{Continuous Characterizations
of $\dbt$ and $\dft$ \label{s3.1}}

Let $\vz\in\cs(\rn)$ be such that
\begin{equation}\label{3.1}
\supp\widehat{\vz}\subset \{\xi\in\rn:\ 1/2\le|\xi|\le 2\}\
\mathrm{and}\ |\widehat{\vz}(\xi)|\ge C>0\ \mathrm{if}\
3/5\le|\xi|\le 5/3.
\end{equation}
Recall that the homogeneous Besov-type space $\dbt$ and
Triebel-Lizorkin-type space $\dft$ are defined as follows; see
\cite{yy1,yy2}.

\begin{definition}\label{d3.1}
Let $s\in\rr$, $\tau\in[0,\fz)$, $q\in(0,\,\fz]$ and
$\vz\in\cs(\rn)$ satisfy \eqref{3.1}.

{\rm (i)} The {\it Besov-type space $\dbt$} with $p\in(0,\fz]$
is defined to be the set of
all $f\in\cs'_\fz(\rn)$ such that $\|f\|_{\dbt}:= \|\{2^{js}
(\vz_j\ast f) \}_{j\in\zz}\|_{\ell^q(L^p_\tau(\rn,\zz))}<\fz$.

{\rm (ii)} The {\it Triebel-Lizorkin-type space $\dft$} with $p\in
(0,\fz)$ is defined to be the set of all $f\in\cs'_\fz(\rn)$ such
that $\|f\|_{\dft}:= \|\{2^{js} (\vz_j\ast f)
\}_{j\in\zz}\|_{L^p_\tau(\ell^q(\rn,\zz))}<\fz$.
\end{definition}

Let $\varepsilon\in(0,\infty)$, $R\in\zz_+\cup\{-1\}$ and $\Phi\in\cs(\rn)$
satisfy that
\begin{equation}\label{3.2}
|\widehat{\Phi}(\xi)|>0\hs \mathrm{on}\hs \{\xi\in\rn:\
\varepsilon/2<|\xi|<2\varepsilon\}\quad\mathrm{and}\quad D^\az(\widehat{\Phi})(0)
=0\hs \mathrm{for\ all}\hs
\|\az\|_{\ell^1}\le R.
\end{equation}
Recall that $\Phi_t\ast f$ for $t\in\rr$ are usually called the {\it
local means}; see, for example, \cite{t92}. We characterize the
space $\ft$ as follows.

\begin{theorem}\label{t3.1}
Let $s\in\rr$, $\tau\in[0,\fz)$, $p\in(0,\fz)$, $q\in(0,\fz]$,
$R\in\zz_+\cup\{-1\}$ and $a\in(n/(p\wedge q), \fz)$ such that
$s+n\tau<R+1$ and $\Phi$ be as in \eqref{3.2}. Then the space $\dft$
is characterized by
$$\dft=\{f\in\cs'_R(\rn):\ \|f|\dft\|_i<\fz\},\quad i\in\{1,\cdots,5\},$$
where
\begin{eqnarray*}
\|f|\dft\|_1\ev\sup_{P\in\mathcal{Q}}\frac1{|P|^\tau}
\lf\{\int_P
\lf[\int_0^{\ell(P)}t^{-sq}|\Phi_t*f(x)|^q\dt\r]^{p/q}\,dx\r\}^{1/p},\noz
\end{eqnarray*}
\begin{eqnarray*}
\|f|\dft\|_2\ev\sup_{P\in\mathcal{Q}}\frac1{|P|^\tau}
\lf\{\int_P \lf[\int_0^{\ell(P)}
t^{-sq}|(\Phi_t^*f)_a(x)|^q\dt\r]^{p/q}\,dx\r\}^{1/p},
\end{eqnarray*}
\begin{eqnarray*}
\|f|\dft\|_3\ev\sup_{P\in\mathcal{Q}}\frac1{|P|^\tau}
\lf\{\int_P\lf[\int_0^{\ell(P)}t^{-sq}\int_{|z|<t}
|\Phi_t*f(x+z)|^q\,dz\frac{dt}{t^{n+1}}\r]^{p/q}\,dx\r\}^{1/p},
\end{eqnarray*}
\begin{eqnarray*}
\|f|\dft\|_4\ev\sup_{P\in\mathcal{Q}}\frac1{|P|^\tau}
\lf\{\int_P
\lf[\sum_{k=j_P}^\fz2^{skq}|(\Phi_k^*f)_a(x)|^q\r]^{p/q}\,dx\r\}^{1/p}
\end{eqnarray*}
and
\begin{eqnarray*}
\|f|\dft\|_5\ev\sup_{P\in\mathcal{Q}}\frac1{|P|^\tau}
\lf\{\int_P
\lf[\sum_{k=j_P}^\fz2^{skq}|\Phi_k*f(x)|^q\r]^{p/q}\,dx\r\}^{1/p},
\end{eqnarray*}
with the usual modification made when $q=\fz$; moreover,
when $a\in(2n/(p \wedge q), \fz)$,
$$\dft=\{f\in\cs'_R(\rn):\ \|f|\dft\|_6<\fz\},$$
where
\begin{eqnarray*}
\|f|\dft\|_6\ev\sup_{P\in\mathcal{Q}}\frac1{|P|^\tau}
\lf\{\int_P \lf(\sum_{j=j_P}^\fz
2^{jsq}\lf[\int_\rn\frac{2^{jn}|\Phi_j\ast
f(x+y)|^r}{(1+2^j|y|)^{ar}}
\,dy\r]^{q/r}\r)^{p/q}\,dx\r\}^{1/p}\noz
\end{eqnarray*}
and $r\in (0, p\wedge q)$ satisfying that $ar>2n$. Furthermore,
all $\|\cdot|\dft\|_i$, $i\in\{1,\cdots,6\}$, are equivalent
(quasi-)norms in $\dft$.
\end{theorem}

We call the characterizations of the space $\dft$ via the (quasi-)norm
$\|\cdot|\dft\|_i$, $i\in\{1,\cdots,5\}$, respectively, the
characterizations by the continuous local means, by the continuous
Peetre maximal function of local means, by the tent space (Lusin
area function) associated with local means, by the discrete Peetre
maximal function of local means and by the discrete local means.

Before we prove the above theorem let us give the following remark.

\begin{remark} Notice that Theorem \ref{t3.1}, in the special case when
$p\in (0,\infty)$ and
$\tau = 1/p$, gives new interesting characterizations
for the spaces $\dot{F}^{s}_{\infty,q}(\mathbb{R}^n)$ introduced by Triebel
\cite{t83} and extended by Frazier and Jawerth \cite{fj2} which can partly be
seen as continuous type Carleson measure characterizations. Discrete
counterparts of the above characterizations have been given by Rychkov
\cite{Ry99a}. Notice also that, for all $p\in(0,\fz)$, the spaces $\dot{F}^{0,1/p}_{p,2}(\rn)$
coincide with $\mathop\mathrm{BMO} ({\mathbb R}^n)$.
\end{remark}

To prove Theorem \ref{t3.1}, we need the following conclusion, which
was first observed in \cite{yy3}. For completeness, we give some
details here.

\begin{lemma}\label{l3.1}
Let $s$, $\tau$, $p$ and $q$ be as in Definition \ref{d3.1}. If
$f\in\dft$ or $\dbt$, then there exists a canonical way to find a
representative of $f$ such that $f \in \cs_L'(\rn)$, where
$L:=(-1)\vee \lfloor s+n(\tau-1/p) \rfloor$.
\end{lemma}

\begin{proof}
We only consider the space $\dft$ by similarity. Let $f\in\dft$ and
$\vz\in\cs(\rn)$ satisfy \eqref{3.1}, Then by \cite[Lemma
(6.9)]{fjw}, there exists a function $\psi\in\cs(\rn)$ satisfying
\eqref{3.1} such that
$\sum_{j\in\zz}\overline{\widehat{\vz}(2^j\xi)}\widehat{\psi}(2^j\xi)=1$
for all $\xi\in\rn\setminus\{0\}$. By the Calder\'on reproducing
formula in \cite[Lemma 2.1]{yy1}, we know that $f=\sum_{m\in\zz}
\widetilde{\psi}_m\ast\vz_m\ast f$ in $\cs'_\fz(\rn),$ where and in
what follows, $\widetilde{\psi}(z):=\overline{\psi(-z)}$ for all
$z\in\rn$. From the arguments in \cite[Lemma 4.2]{yy2} (see also
\cite[pp.\,153-155]{fj2} and \cite[Proposition 3.8 and Corollary
3.9]{bh}), we deduce that there exists a sequence
$\{P_N\}_{N\in\nn}$ of polynomials, with degree no more than
$L$ same as in Lemma \ref{l3.1} for all $N\in\nn$,
such that $g:= \lim_{N\to\fz}(\sum_{m=-N}^N
\widetilde{\psi}_m\ast\vz_m\ast f+P_N)$ in $\cs'(\rn)$ and
$g$ is a representative of the equivalence class
$f+\mathcal{P}(\rn)$. We identify $f$ with its representative $g$.
In this sense, $f\in\cs'_L(\rn)$, which completes the proof of
Lemma \ref{l3.1}.
\end{proof}

By \cite[(2.66)]{u10} and the argument in the proof of
\cite[Theorem 2.8]{u10}, we have the following estimate, which is widely used in
this paper.

\begin{lemma}\label{l3.2}
Let $R\in\zz_+\cup\{-1\}$, $\Phi\in\cs(\rn)$ satisfy \eqref{3.2}
and $f\in \cs'_R(\rn)$. For all $t\in[1,2]$, $a\le N$,
$\ell\in\zz$ and $x\in\rn$,
\begin{equation}\label{3.3}
(\Phi_{2^{-\ell}t}^*f)_a(x)^r\le C(r)\sum_{k=0}^\fz2^{-kNr}2^{(k+\ell)n}
\int_\rn\frac{|(\Phi_{k+\ell})_t*f(y)|^r}{(1+2^\ell|x-y|)^{ar}}\,dy,
\end{equation}
where $r$ is an arbitrary fixed positive number and $C(r)$ is a
positive constant independent of $\Phi$, $f$, $\ell$, $x$ and $t$,
but may depend on $r$.
\end{lemma}

Now we turn to the proof of  Theorem \ref{t3.1}.

\begin{proof}[Proof of  Theorem \ref{t3.1}]
We prove Theorem \ref{t3.1} in four steps.
First we show that for the same $\Phi$,
$\|f|\dft\|_i$, $i\in\{1,2,4,5\}$, are equivalent each other.
Next we prove that $\|f|\dft\|_i$, $i\in\{1,2,4,5\}$, are independent
of the choice of $\Phi$ and the conclusions of Theorem \ref{t3.1}
for $i\in\{1,2,4,5\}$ hold. Finally, for the same $\Phi$, we show that
$\|f|\dft\|_2\sim\|f|\dft\|_3$ and $\|f|\dft\|_5\sim\|f|\dft\|_6$.

\emph{Step 1.} In this step, for the same $\Phi$,
we prove the relations that
\begin{equation}\label{3.4}
\|f|\dft\|_1\sim\|f|\dft\|_2\sim\|f|\dft\|_4\sim\|f|\dft\|_5
\end{equation}
for every $f\in\cs'_R(\rn)$.

Notice that obviously, by the definitions, we see
that $\|f|\dft\|_1\le\|f|\dft\|_2$
and $\|f|\dft\|_5\le\|f|\dft\|_4$. Next we show that
$\|f|\dft\|_2\ls\|f|\dft\|_1$.

We choose $r\in({n}/{a},p\wedge q).$
Then by \eqref{3.3}, we have
\begin{eqnarray*}
\|f|\dft\|_2&&\le\sup_{P\in\mathcal{Q}}\frac1{|P|^\tau}
\lf\{\int_P\lf[\sum_{\ell=j_P}^\fz\int_{2^{-\ell}}^{2^{-\ell+1}}
t^{-sq}|(\Phi_t^*f)_a(x)|^q\dt\r]^{p/q}\,dx\r\}^{1/p}\noz\\
&&\ls\sup_{P\in\mathcal{Q}}\frac1{|P|^\tau}
\lf\{\int_P\lf(\sum_{\ell=j_P}^\fz\int_{1}^{2}
2^{\ell sq}\lf[\sum_{k=0}^\fz2^{-kNr}2^{(k+\ell)n}\r.\r.\r.\noz\\
&&\hs\times\lf.\lf.\lf.\int_\rn\frac{|(\Phi_{k+\ell})_t*f(y)|^r}
{(1+2^\ell|x-y|)^{ar}}\,dy\r]^{q/r}\dt\r)^{p/q}\,dx\r\}^{1/p},\noz
\end{eqnarray*}
where the natural number $N\in[a,\fz)$ is determined later.
From Minkowski's inequality, it follows that
\begin{eqnarray*}
\|f|\dft\|_2
&&\ls\sup_{P\in\mathcal{Q}}\frac1{|P|^\tau}
\lf\{\int_P\lf(\sum_{\ell=j_P}^\fz
2^{\ell sq}\left[\sum_{k=0}^\fz2^{-kNr}2^{(k+\ell)n}\r.\r.\r.\noz\\
&&\hs\times\lf.\lf.\lf.\int_\rn\frac{
\left[\int_1^2|(\Phi_{k+\ell})_t*f(y)|^q\,\frac{dt}t\r]^{r/q}}
{(1+2^\ell|x-y|)^{ar}}\,dy\r]^{q/r}\r)^{p/q}\,dx\r\}^{1/p}.\noz
\end{eqnarray*}

Fix any $P\in\cq$. Notice that $\ell\ge j_P$ and
\begin{equation}\label{3.5}
1+2^\ell|x-y|\gs 2^\ell 2^{-j_P}\|i\|_{\ell^1}
\end{equation}
for all $x\in P$ and $y\in P+i\ell(P)$
with $i\in\zz_+^n$ and $\|i\|_{\ell^1}\ge2$.
From these facts, we then infer that, for
all $x\in P$,
\begin{eqnarray}\label{3.6}
&&\int_\rn\frac{\left[\int_1^2|(\Phi_{k+\ell})_t*f(y)|^q\,\frac{dt}t\r]^{r/q}}
{(1+2^\ell|x-y|)^{ar}}\,dy\noz\\
&&\hs\ls\int_{3P}\frac{\left[\int_1^2|(\Phi_{k+\ell})_t*f(y)|^q\,\frac{dt}t\r]^{r/q}}
{(1+2^\ell|x-y|)^{ar}}\,dy+\sum_{{i\in\zz_+^n},\,{\|i\|_{\ell^1}\ge2}}
\int_{P+i\ell(P)}\cdots\noz\\
&&\hs\ls2^{-\ell n}
M\left(\left[\int_1^2|(\Phi_{k+\ell})_t*f|^q\,\frac{dt}t\r]^{r/q}\chi_{3P}\r)(x)
+\sum_{{i\in\zz_+^n},\,{\|i\|_{\ell^1}\ge2}}\|i\|_{\ell^1}^{-ar}2^{-\ell
ar}2^{j_Par}\noz\\
&&\hs\hs\times\int_{P+i\ell(P)}{\left[\int_1^2|(\Phi_{k+\ell})_t*f(y)|^q
\,\frac{dt}t\r]^{r/q}}
\chi_{P+i\ell(P)}(y)\,dy\noz\\
&&\hs\ls2^{-\ell n}M\left(\left[\int_1^2|(\Phi_{k+\ell})_t*f|^q
\,\frac{dt}t\r]^{r/q}\chi_{3P}\r)(x)
+\sum_{{i\in\zz_+^n},\,{\|i\|_{\ell^1}\ge2}}\|i\|_{\ell^1}^{-ar}2^{-\ell
ar}\noz\\
&&\hs\hs\times2^{j_P(ar-n)}
M\left(\left[\int_1^2|(\Phi_{k+\ell})_t*f|^q\,
\frac{dt}t\r]^{r/q}\chi_{P+i\ell(P)}\r)(x+i\ell(P))
=:{\mathrm I}_1+{\mathrm I}_2.
\end{eqnarray}

For the term ${\mathrm I}_1$, letting $\delta>0$ and $N>\delta\vee (\delta+n/r-s)$,
from H\"older's inequality and the Fefferman-Stein vector-valued
inequality \eqref{2.3}, we deduce that
\begin{eqnarray}\label{3.7}
&&\sup_{P\in\mathcal{Q}}\frac1{|P|^\tau}
\lf\{\int_P\lf[\sum_{\ell=j_P}^\fz2^{\ell
sq}\lf(\sum_{k=0}^\fz2^{-kNr}2^{(k+\ell)n}{\mathrm I}_1\r)^{q/r}
\r]^{p/q}\,dx\r\}^{1/p}\noz\\
&&\hs\ls\sup_{P\in\mathcal{Q}}\frac1{|P|^\tau}
\lf\{\int_P\lf[\sum_{\ell=j_P}^\fz
\sum_{k=0}^\fz2^{-k(N-\dz)q+knq/r}2^{-k sq}\r.\r.\noz\\
&&\hs\hs\times\lf.\lf.
\lf[M\lf(\left[\int_{2^{-k-\ell}}^{2^{-k-\ell+1}}
t^{-sq}|\Phi_t*f|^q\,
\frac{dt}t\r]^{r/q}\chi_{3P}\r)(x)
\r]^{q/r}\r]^{p/q}\,dx\r\}^{1/p}\noz\\
&&\hs\ls\sup_{P\in\mathcal{Q}}\frac1{|P|^\tau}
\lf\{\int_{3P}\lf[\sum_{\ell=j_P}^\fz
\sum_{k=0}^\fz2^{-k(N-\dz)q+knq/r}2^{-k sq}\r.\r.\noz\\
&&\hs\hs\times\lf.\lf.
\int_{2^{-k-\ell}}^{2^{-k-\ell+1}}
t^{-sq}|\Phi_t*f(x)|^q\,
\frac{dt}t
\r]^{p/q}\,dx\r\}^{1/p}\ls\|f|\dft\|_1.
\end{eqnarray}
Similar to the estimate \eqref{3.7}, for the term ${\mathrm I}_2$,
by using $ar>n$, we also conclude that
\begin{eqnarray}\label{3.8}
\sup_{P\in\mathcal{Q}}\frac1{|P|^\tau}
\lf\{\int_P\lf[\sum_{\ell=j_P}^\fz2^{\ell
sq}\lf(\sum_{k=0}^\fz2^{-kNr}2^{(k+\ell)n} {\mathrm I}_2\r)^{q/r}
\r]^{p/q}\,dx\r\}^{1/p}\ls\|f|\dft\|_1.
\end{eqnarray}
Combining the estimate \eqref{3.7} and \eqref{3.8},
we see that
$\|f|\dft\|_2\ls\|f|\dft\|_1.$

With slight modifications of the above argument, we also obtain
$$\|f|\dft\|_2\ls\|f|\dft\|_5\quad\mathrm{and}\quad
\|f|\dft\|_4\ls\|f|\dft\|_1,$$
which yields \eqref{3.4} and completes the proof of Step 1.

\emph{Step 2}. In this step, we first show that $\|f|\dft\|_i$,
$i\in\{1,2,4,5\}$, are independent of the choice of $\Phi$.
To this end, we temporarily write $\|f|\dft\|_4$ in Theorem \ref{t3.1}
by $\|f|\dft\|^\Phi_4$.
Let $\Psi$ also satisfy \eqref{3.2} and we use $\|f|\dft\|^\Psi_4$
to denote $\|f|\dft\|_4$ as in Theorem \ref{t3.1} via replacing
$\Phi$ therein by $\Psi$. By Step 1 and the symmetry, we see that
it suffices to prove that
$\|f|\dft\|^\Psi_4\ls \|f|\dft\|^\Phi_4.$

Indeed, by \cite[(2.89)]{u10}, we know that for all $\ell\in\zz$,
\begin{eqnarray*}
2^{\ell s}(\Psi^\ast_\ell f)_a\ls
\sum_{k=0}^\ell 2^{(k-\ell)(R+1-s)}2^{ks}(\Phi^\ast_k f)_a
+\sum_{k=\ell+1}^\infty2^{(\ell-k)(K+1-a+s)}2^{ks}(\Phi^\ast_k f)_a,
\end{eqnarray*}
where $K\in \nn$ is sufficiently large. Choosing
$K>a+|s|+n\tau$, by Lemma \ref{l2.1}, we immediately conclude that
$\|f|\dft\|^\Psi_4\ls \|f|\dft\|^\Phi_4,$ which, combined with
\eqref{3.4}, further implies that $\|f|\dft\|_i$, $i\in\{1,2,4,5\}$,
are independent of the choice of $\Phi$.

Next we show that the conclusions of Theorem \ref{t3.1} for
$i\in\{1,2,4,5\}$ hold. Notice that if $\vz$ satisfies
\eqref{3.1}, then it also satisfies
\eqref{3.2} for $\varepsilon=1$.
From this observation and the definition of $\|f\|_\dft$, together
with the independence of $\|f|\dft\|_5$ of $\Phi$, it further follows
that $\|f|\dft\|_5\sim\|f|\dft\|_5^\vz\sim\|f\|_\dft$.
To complete the proof of Step 2, since $\cs'_R(\rn)\subset \cs'_\infty(\rn)$,
we see that it suffices to show that if $f\in\dft$, then $f\in\cs'_R(\rn)$. Indeed,
by Lemma \ref{l3.1}, we know that $f\in\cs'_L(\rn)$ with
$L:=(-1)\vee\lfloor s+n(\tau-1/p) \rfloor$, which together
with the fact $L\le R$, implies that $f\in\cs'_R(\rn)$. Thus,
$$\dft=\{f\in\cs'_R(\rn):\ \|f|\dft\|_i<\fz\},\quad i\in\{1,2,4,5\},$$
which completes the proof of Step 2.

\emph{Step 3}. In this step, for a fixed $\Phi$,
we prove that $\|f|\dft\|_2\sim\|f|\dft\|_3$ for all
$f\in\cs'_R(\rn)$. Since the inequality
$\|f|\dft\|_3\ls\|f|\dft\|_2$ is trivial, we only need to show
$\|f|\dft\|_2\ls\|f|\dft\|_3$.

Notice that for all $k\ge0$ and $\ell\in\zz$,
when $t\in[1,2]$ and $|z|<2^{-(k+\ell)}t$,
$$1+2^\ell|x-y|\le1+2^\ell(|x-(y+z)|+|z|)\ls 1+2^\ell|x-(y+z)|.$$
By this and \eqref{3.3}, we conclude that,
for all $t\in[1,2]$, $a\le N$, $\ell\in\zz$ and $x\in\rn$,
\begin{eqnarray*}
\lf[(\Phi_{2^{-\ell}t}^*f)_a(x)\r]^r
&&\ls\sum_{k=0}^\fz2^{-kNr+2(k+\ell)n}\int_\rn
\int_{|z|<2^{-(k+\ell)} t}
\frac{|(\Phi_{k+\ell})_t*f(y)|^r}{(1+2^\ell|x-y|)^{ar}}\,dz\,dy\\
&&\ls \sum_{k=0}^\fz2^{-kNr+2(k+\ell)n}\int_\rn\int_{|z|<2^{-(k+\ell)}t}
\frac{|(\Phi_{k+\ell})_t*f(y+z)|^r}{(1+2^\ell|x-y|)^{ar}}\,dz\,
dy.
\end{eqnarray*}
Thus, from Minkowski's inequality, we further deduce that
\begin{eqnarray}\label{3.9}
&&\int_1^2 \lf[(\Phi_{2^{-\ell}t}^*f)_a(x)\r]^q \frac{dt}t\noz\\
&&\hs\ls  \lf\{\sum_{k=0}^\fz2^{-kNr+2(k+\ell)n}\int_\rn
\frac{\left[\int_1^2\int_{|z|<2^{-(k+\ell)}t}
|(\Phi_{k+\ell})_t*f(y+z)|^q\,dz\,\frac{dt}t\r]^{r/q}}{(1+2^\ell|x-y|)^{ar}}\,
dy\r\}^{q/r},\qquad
\end{eqnarray}
which, together with
\begin{eqnarray}\label{3.10}
\|f|\dft\|_2
\ls\sup_{P\in\cq}\frac1{|P|^\tau}\lf\{\int_P\lf[
\sum_{\ell=j_P}^\fz2^{\ell sq}\int_{1}^{2}
[(\Phi_{2^{-\ell}t}^*f)_a(x)]^q\dt\r]^{\frac pq}dx\r\}^{\frac 1p}
\end{eqnarray}
and H\"older's inequality, implies that
\begin{eqnarray*}
&&\|f|\dft\|_2\\
&&\hs\ls\sup_{P\in\cq}\frac1{|P|^\tau}\lf\{\int_P\lf[
\sum_{\ell=j_P}^\fz2^{\ell sq+2\ell nq/r}
\sum_{k=0}^\fz2^{-k(N-\dz)q+2knq/r}\r.\r.\noz\\
&&\hs\hs\times\lf.\lf.\lf\{\int_\rn
\frac{\left[\int_1^2\int_{|z|<2^{-(k+\ell)}t}
|(\Phi_{k+\ell})_t*f(y+z)|^q\,dz\,\frac{dt}t\r]^{r/q}}{(1+2^\ell|x-y|)^{ar}}\,
dy\r\}^{q/r}\r]^{p/q}dx\r\}^{1/p}\noz,
\end{eqnarray*}
where $\delta\in(0,N)$.

By \eqref{3.5} and $ar>n$, we conclude that, for all $x\in P$,
\begin{eqnarray*}
&&\int_\rn \frac{\left[\int_1^2\int_{|z|<2^{-(k+\ell)}t}
|(\Phi_{k+\ell})_t*f(y+z)|^q\,dz\,\frac{dt}t\r]^{r/q}}
{(1+2^\ell|x-y|)^{ar}}\,dy\noz\\
&&\hs\le\int_{3P} \frac{\left[\int_1^2\int_{|z|<2^{-(k+\ell)}t}
|(\Phi_{k+\ell})_t*f(y+z)|^q\,dz\,\frac{dt}t\r]^{r/q}}
{(1+2^\ell|x-y|)^{ar}}\,dy
+\sum_{\|i\|_{\ell^1}\ge2,i\in\zz^n}\int_{P+i\ell(P)}\cdots\noz\\
&&\hs\ls 2^{-\ell n}M\lf( \left[\int_1^2\int_{|z|<2^{-(k+\ell)}t}
|(\Phi_{k+\ell})_t*f(\cdot+z)|^q\,dz\,\frac{dt}t\r]^{r/q}\chi_{3P}
\r)(x)\noz\\
&&\hs\hs+\sum_{\|i\|_{\ell^1}\ge2,\,i\in\zz^n_+}\|i\|_{\ell^1}^{-ar}
2^{-(\ell-j_P)(ar-n)}2^{-\ell n}\\
&&\hs\hs\times M\lf(\left[\int_1^2\int_{|z|<2^{-(k+\ell)}t}
|(\Phi_{k+\ell})_t*f(\cdot+z)|^q\,dz\,\frac{dt}t\r]^{r/q}
\chi_{P+i\ell(P)}\r)(x+i\ell(P))\noz.
\end{eqnarray*}
Then applying the Fefferman-Stein vector-valued inequality,
by an argument similar to the estimate of \eqref{3.7}, we further see that
$\|f|\dft\|_2\ls\|f|\dft\|_3$, which completes
the proof of Step 3.

\emph{Step 4}. Finally, in this step, we prove that
$\|f|\dft\|_5\sim\|f|\dft\|_6$.
To this end, by \cite[(2.48)]{u10}, we know that for all
$\ell\in\zz$ and $x\in\rn$,
\begin{eqnarray*}
|\Phi_\ell*f(x)|^r\ls\sum_{k=0}^\fz2^{-kNr}2^{(k+\ell)n}
\int_\rn\frac{|\Phi_{k+\ell}*f(x+y)|^r}{(1+2^\ell|y|)^{Nr}}\,dy,
\end{eqnarray*}
where $N\in\nn$ is sufficiently large and determined later.
Letting $N>a-s$, we have
\begin{eqnarray*}
2^{\ell sr}|\Phi_\ell*f(x)|^r\ls
 \sum_{i=\ell}^\fz2^{-(i-\ell)(N-a+s)r}
\int_\rn\frac{2^{in}|\Phi_i*f(x+y)|^r}{(1+2^i|y|)^{ar}}\,dy,
\end{eqnarray*}
which, combined with Lemma \ref{l2.1}, further implies that
$\|f|\dft\|_5\ls\|f|\dft\|_6$.

Conversely, from $a/2>n/(p\wedge q)$ and $ar/2>n$,
we infer that, for all $x\in\rn$,
\begin{eqnarray*}
\int_\rn\frac{2^{jn}|\Phi_j*f(x+y)|^r}{(1+2^j|y|)^{ar}}\,dy
\ls\sup_{z\in\rn} \frac{|\Phi_j*f(x+z)|^r}{(1+2^j|z|)^{ar/2}}
\int_\rn\frac{2^{jn}}{(1+2^j|y|)^{ar/2}}\,dy
\sim(\Phi_j^*f)_{a/2}^r(x).
\end{eqnarray*}
By this, Steps 2 and 3, we know that
$\|f|\dft\|_6\ls\|f|\dft\|_5$ immediately, which completes the proof
of Theorem \ref{t3.1}.
\end{proof}

The Besov-type spaces also have the characterizations as in Theorem \ref{t3.1},
whose proofs are similar to that of Theorem \ref{t3.1}.
We omit the details.

\begin{theorem}\label{t3.2}
Let $s\in\rr$, $\tau\in[0,\fz)$, $p,\,q\in(0,\fz]$,
$R\in\zz_+\cup\{-1\}$ and $a\in(n/p, \fz)$ such that $s+n\tau<R+1$
and $\Phi$ be as in \eqref{3.2}. Then the space $\dbt$ is
characterized by
$$\dbt=\{f\in\cs'_R(\rn):\ \|f|\dbt\|_i<\fz\},\quad i\in\{1,\cdots,4\},$$
where
\begin{eqnarray*}
\|f|\dbt\|_1 \ev\sup_{P\in\mathcal{Q}}\frac1{|P|^\tau}
\lf\{\int_0^{\ell(P)}t^{-sq}
\lf[\int_P|\Phi_t*f(x)|^p\,dx\r]^{q/p}\dt\r\}^{1/q},\noz
\end{eqnarray*}
\begin{eqnarray*}
\|f|\dbt\|_2
\ev\sup_{P\in\mathcal{Q}}\frac1{|P|^\tau}
\lf\{\int_0^{\ell(P)}t^{-sq}
\lf[\int_P|(\Phi_t^*f)_a(x)|^p\,dx\r]^{q/p}\dt\r\}^{1/q},\noz
\end{eqnarray*}
\begin{eqnarray*}
\|f|\dbt\|_3\ev\sup_{P\in\mathcal{Q}}\frac1{|P|^\tau}
\lf\{\sum_{k=j_P}^\fz2^{skq}
\lf[\int_P|(\Phi_k^*f)_a(x)|^p\,dx\r]^{q/p}\r\}^{1/q}
\end{eqnarray*}
and
\begin{eqnarray*}
\|f|\dbt\|_4\ev\sup_{P\in\mathcal{Q}}\frac1{|P|^\tau}
\lf\{\sum_{k=j_P}^\fz2^{skq}
\lf[\int_P|\Phi_k*f(x)|^p\,dx\r]^{q/p}\r\}^{1/q},
\end{eqnarray*}
with the usual modifications made when $p=\fz$ or $q=\fz$.
Moreover, when $a\in(2n/p, \fz)$,
$$\dbt=\{f\in\cs'_R(\rn):\ \|f|\dbt\|_5<\fz\},$$
where
\begin{eqnarray*}
\|f|\dbt\|_5\ev\sup_{P\in\mathcal{Q}}\frac1{|P|^\tau}
\lf\{\sum_{j=j_P}^\fz 2^{jsq}\lf(\int_P
\lf[\int_\rn\frac{2^{jn}|\Phi_j\ast
f(x+y)|^r}{(1+2^j|y|)^{ar}}
\,dy\r]^{p/r}\,dx\r)^{q/p}\r\}^{1/q}\noz
\end{eqnarray*}
and $r\in (0, p)$ satisfying that $ar>2n$.
Furthermore, all
quantities $\|f|\dbt\|_i$, $i\in\{1,\cdots,5\}$, are equivalent
(quasi-)norms in $\dbt$.
\end{theorem}

With slight modifications for the proofs of Theorems \ref{t3.1} and
\ref{t3.2}, we obtain the following observation. We omit the details.

\begin{remark}\label{r3.1}\rm
The conclusions of Theorems \ref{t3.1} and \ref{t3.2} are still true if, for all
$f\in\cs'_R(\rn)$ and $x\in\rn$, we
replace $(\Phi_t^*f)_a(x)$ by
$$(\widetilde{\Phi}_t^*f)_a(x):=\sup_{\gfz{y\in\rn}{t/2\le s\le
t}}\frac{|\Phi_s\ast f(x+y)|}{(1+|y|/s)^a}.$$
\end{remark}

The most prominent example $\Phi$ of functions satisfying
\eqref{3.2} is the classical local means; see \cite[Section
3.3]{t92} for examples. In particular, we have the following statements.

\begin{corollary}\label{c3.1}
Let $k\in \cs(\rn)$ such that $\widehat{k}(0)\neq 0$ and
$\Psi:= \Delta^Nk$ with $N\in\nn$ and $2N>s+n\tau$. Then,

{\rm (i)} For $p,\,q,\,s$ and $\tau$ as in Theorem \ref{t3.1}
and $\Phi$ therein replaced by $\Psi$, the conclusions of
Theorem \ref{t3.1} are true.

{\rm (ii)} For $p,\,q,\,s$ and $\tau$ as in Theorem \ref{t3.2}
and $\Phi$ therein replaced by $\Psi$, the conclusions of
Theorem \ref{t3.2} are true.
\end{corollary}

\begin{corollary}\label{c3.2}
Let $\varphi_0\in\cs(\rn)$ be a non-increasing radial function satisfying
$\vz_0(0)\neq0$ and $D^{\az}\vz_0(0)=0$
for $1\le\|\az\|_{\ell^1}\le R$ with $R+1>s+n\tau$. Define
$\vz(\cdot)\ev\vz_0(\cdot)-\vz_0(2\cdot)$ and $\Psi:=\mathcal{F}^{-1}\vz$.

{\rm (i)} For $p,\,q,\,s$ and $\tau$ as in Theorem \ref{t3.1}
and $\Phi$ therein replaced by $\Psi$, the conclusions of
Theorem \ref{t3.1} are true.

{\rm (ii)} For $p,\,q,\,s$ and $\tau$ as in Theorem \ref{t3.2}
and $\Phi$ therein replaced by $\Psi$, the conclusions of
Theorem \ref{t3.2} are true.
\end{corollary}

\subsection{Continuous Characterizations of $\dbh$ and
$\dfh$\label{s3.2}}

In this subsection, we focus on the continuous local
mean characterizations of homogeneous Besov-Hausdorff
spaces $\dbh$ and Triebel-Lizorkin-Hausdorff spaces
$\dfh$.

We begin with the notions of Besov-Hausdorff and
Triebel-Lizorkin-Hausdorff spaces introduced in \cite{ysy}.

\begin{definition}\label{d3.2}
Let $s\in\rr$, $p\in(1,\fz)$ and $\vz\in\cs(\rn)$ satisfy
\eqref{3.1}.

{\rm (i)} The {\it Besov-Hausdorff space $\dbh$} with $q\in[1,\fz)$ and
$\tau\in[0, 1/{(p\vee q)'}]$ is defined to be the set of all
$f\in\cs_\fz'(\rn)$ such that
$$\|f\|_{\dbh}:= \|\{2^{js} (\vz_j\ast
f)\}_{j\in\zz}\|_{\widetilde{\ell^q(L^p_\tau(\rn,\zz))}}<\fz.$$

{\rm (ii)} The {\it Triebel-Lizorkin-Hausdorff space $\dfh$} with $q\!\in\!
(1,\fz)$ and $\tau\!\in\![0, 1/{(p\vee q)'}]$
is defined to be the
set of all $f\in\cs_\fz'(\rn)$ such that
$$\|f\|_{\dfh}:= \|\{2^{js} (\vz_j\ast f)
\}_{j\in\zz}\|_{\widetilde{L^p_\tau(\ell^q(\rn,\zz))}}<\fz.$$
\end{definition}

Recall that $B\dot{H}_{p,q}^{s,0}(\rn)=
\dot{B}_{p,q}^{s}(\rn)$, $F \dot{H}_{p,q}^{s,0}(\rn)=
\dot{F}_{p,q}^{s}(\rn)$ and $F{\dot
H}_{2,2}^{-\alpha,1/2-\alpha}(\rn)= HH^1_{-\alpha}(\rn)$ (see
\cite{yy1}), where the Hardy-Hausdorff space $HH^1_{-\alpha}(\rn)$
was recently introduced by Dafni and Xiao in \cite{dx} and proved to be the
predual space of $Q_\alpha(\rn)$ therein. More applications
of $Q_\alpha(\rn)$ and $HH^1_{-\alpha}(\rn)$ were given by Dafni and Xiao 
in \cite{dx05}. We also point out that Adams and Xiao \cite{Ax04}
and Dafni and Xiao \cite{dx} first used the Hausdorff capacity and weights
to introduce some new spaces of functions and, very recently, 
Adams and Xiao \cite{Ax11, Ax12-1, Ax12-2} gave some new 
interesting applications of such spaces of functions.

\begin{remark}\label{r3.2}
{\rm (i)} The Besov-Hausdorff space $\dbh$ and
Triebel-Lizorkin-Hausdorff space $\dfh$ are
quasi-Banach spaces; see \cite{yy1,yy2,ysy}.
Indeed, by \cite[Remarks 7.1 and 7.3]{ysy}, we know that,
for all $f_1,\,f_2\in \dfh$,
$$\|f_1+f_2\|_{\dfh}\le 2^{(p\vee q)'}\lf[\|f_1\|_{\dfh}+\|f_2\|_{\dfh}\r].$$
Recall that $(p\vee q)'$ denotes the conjugate index of $p\vee q$.
An inequality similar to this is also true for $\dbh$.

{\rm (ii)} By the Aoki-Rolewicz theorem (see \cite{ao,Ro57}), there exists
$v\in(0,1]$ such that
\begin{equation}\label{3.11}
\left\|\sum_{j\in\zz} f_j\right\|_{\dbh}^v
\lesssim
\sum_{j\in\zz} \left\|f_j\right\|_{\dbh}^v\quad \mathrm{for\ all}\
\{f_j\}_{j\in\zz}\subset \dbh,
\end{equation}
and
\begin{equation}\label{3.12}
\left\|\sum_{j\in\zz} f_j\right\|_{\dfh}^v
\lesssim
\sum_{j\in\zz} \left\|f_j\right\|_{\dfh}^v \quad \mathrm{for\ all}\
\{f_j\}_{j\in\zz}\subset \dfh.
\end{equation}
Indeed, $v:=1/({1+(p\vee q)'})$ does the job.
\end{remark}

We now characterize $\dfh$ and $\dbh$ via the local means as follows.

\begin{theorem}\label{t3.3}
Let $s\in\rr$, $p,\,q\in(1,\fz)$,
$\tau\in[0, 1/{(p\vee q)'}]$, $R\in\zz_+\cup\{-1\}$ and
$a\in(n(1/(p\wedge q)+\tau), \fz)$ such that $s+n\tau<R+1$
and $\Phi$ be as in \eqref{3.2}.
Then the space $\dfh$ is
characterized by
$$\dfh=\{f\in\cs'_R(\rn):\ \|f|\dfh\|_i<\fz\},\quad i\in\{1,\cdots,5\},$$
where
\begin{eqnarray*}
\|f|\dfh\|_1\ev\inf_\oz\lf\|\lf\{\int_0^\fz t^{-sq}|\Phi_t*f(\cdot)|^q[\oz
(\cdot,t)]^{-q}\dt\r\}^{1/q}\r\|_\lp,\noz
\end{eqnarray*}
\begin{eqnarray*}
\|f|\dfh\|_2\ev\inf_\oz\lf\|\lf\{\int_0^\fz t^{-sq}[(\Phi_t^*f)_a(\cdot)]^q[\oz
(\cdot,t)]^{-q}\dt\r\}^{1/q}\r\|_\lp,\noz
\end{eqnarray*}
\begin{eqnarray*}
\|f|\dfh\|_3\ev
\inf_\oz\lf\|\lf\{\int_0^\fz t^{-sq}\int_{|z|<t}|\Phi_t*f(\cdot+z)|^q[\oz
(\cdot+z,t)]^{-q}\,dz\frac{dt}{t^{n+1}}\r\}^{1/q}\r\|_\lp,
\end{eqnarray*}
\begin{eqnarray*}
\|f|\dfh\|_4\ev \|\{2^{ks}\Phi_k*f\}_{k\in\zz}\|_{\wz{L^p_\tau
(\ell^q(\rn,\zz))}}
\end{eqnarray*}
and
\begin{eqnarray*}
\|f|\dfh\|_5\ev
\|\{2^{ks}(\Phi_k^*f)_a\}_{k\in\zz}\|_{\wz{L^p_\tau (\ell^q(\rn,\zz))}},
\end{eqnarray*}
where the infimums are taken over all
nonnegative Borel measurable functions $\omega$ on $\R_+^{n+1}$
satisfying \eqref{2.2}; moreover, when $a\in(2n(1/(p\wedge q)+\tau), \fz)$,
$$\dfh=\{f\in\cs'_R(\rn):\ \|f|\dft\|_6<\fz\},$$
where
\begin{eqnarray*}
\|f|\dfh\|_6\ev\inf_\oz\lf\|\lf\{\sum_{j\in\zz} 2^{jsq}\lf[\int_\rn\frac{2^{jn}|\Phi_j\ast
f(\cdot+y)|^r}{(1+2^j|y|)^{ar}}
\,dy\r]^{\frac qr}[\oz
(\cdot,t)]^{-q}\r\}^{\frac 1q}\r\|_\lp,\noz
\end{eqnarray*}
where $\omega$ runs over all nonnegative Borel measurable
functions on $\R_+^{n+1}$ satisfying \eqref{2.2} and
$r\in (0, p\wedge q)$ satisfying that $(a-n\tau)r>2n$.
Furthermore, all
$\|\cdot|\dfh\|_i$, $i\in\{1,\cdots,6\}$, are equivalent
(quasi-)norm in $\dfh$.
\end{theorem}

To prove this theorem, we need the following two technical lemmas.
The first one is \cite[Lemma 3.2]{ysay}, which reflects the
geometrical properties of Hausdorff capacities.

\begin{lemma}\label{l3.3}
Let $\beta\in[1,\infty)$, $\lambda\in (0,\infty)$
and $\omega$ be a nonnegative Borel measurable function
on $\rr^{n+1}_+$. Then there exists a positive
constant $C$, independent of $\beta$, $\omega$ and $\lambda$, such that
$H^d(\{x\in\rn:\ N_\beta\omega(x)>\lambda\})
\le C \beta^d H^d(\{x\in\rn:\ N\omega(x)>\lambda\}),$
where $N_\beta\omega(x):=\sup_{|y-x|<\beta t}\omega(y, t)$
for all $x\in\rn$.
\end{lemma}

As a counterpart of Lemma \ref{l3.1}, we have the following result,
which was implicitly contained in the proof of \cite[Lemma 3.2]{yy3}.
For the convenience of the reader, we give some details here.

\begin{lemma}\label{l3.4}
Let $s$, $\tau$, $p$ and $q$ be as in Definition \ref{d3.2}.
If $f\in\dfh$ or $\dbh$, then there exists a canonical way to find a
representative of $f$ such that $f \in \cs_L'(\rn)$, where
$L:=(-1)\vee \lfloor s-n(\tau+1/p) \rfloor$.
\end{lemma}

\begin{proof}
We only consider the space $\dfh$ by similarity.
Let $f\in\dfh$ and $\vz\in\cs(\rn)$ satisfy \eqref{3.1},
By the proof of Lemma \ref{l3.1},
there exists a function $\psi\in\cs(\rn)$
satisfying \eqref{3.1} such that $f=\sum_{m\in\zz}
\widetilde{\psi}_m\ast\vz_m\ast f$ in $\cs'_\fz(\rn).$
From the arguments in \cite[Lemma 3.4]{ysay},
we deduce that there exists a sequence
$\{P_N\}_{N\in\nn}$ of polynomials, with degree not more than
$L$ for all
$N\in\nn$, such that $g:= \lim_{N\to\fz}(\sum_{m=-N}^N
\widetilde{\psi}_m\ast\vz_m\ast f+P_N)$ exists in $\cs'(\rn)$ and
$g$ is a representative of the equivalence class
$f+\mathcal{P}(\rn)$. We identify $f$
with its representative $g$. In this sense,
$f\in\cs'_L(\rn)$, which completes the proof of
Lemma \ref{l3.4}.
\end{proof}

Now we are ready to prove Theorem \ref{t3.3}.

\begin{proof}[Proof of Theorem \ref{t3.3}]
We first prove that for all $f\in\cs'_R(\rn)$,
\begin{eqnarray*}
\|f|\dfh\|_1&&\sim\|f|\dfh\|_2\sim\|f|\dfh\|_3\\
&&\sim\|f|\dfh\|_4\sim\|f|\dfh\|_5.
\end{eqnarray*}
By similarity, we only give the details for the first and second equivalences.
We begin with comparing the continuous local means and the continuous Peetre
maximal function of the local means.

Obviously, $\|f|\dfh\|_1\ls\|f|\dfh\|_2$. Next we show
$$\|f|\dfh\|_2\ls\|f|\dfh\|_1.$$

Let $\wz \oz$ be a nonnegative function on $\rr^{n+1}_+$
such that
\begin{eqnarray*}
\lf\|\lf\{\int_0^\fz t^{-sq}|\Phi_t*f(\cdot)|^q
[\wz\oz(\cdot,t)]^{-q}\dt\r\}^{1/q}\r\|_\lp\ls\|f|\dfh\|_1\noz.
\end{eqnarray*}
Notice that
\begin{eqnarray*}
\|f|\dfh\|_2
\sim\inf_\oz\lf\|\lf\{\sum_{\ell\in\zz}2^{\ell sq}\int_1^2
\lf[(\Phi_{2^{-\ell}t}^*f)_a(\cdot)\r]^q[\oz(\cdot,t)]^{-q}\dt\r\}^{1/q}\r\|_\lp.\noz
\end{eqnarray*}
Then by \eqref{3.3} and Minkowski's inequality, we see that
\begin{eqnarray}\label{3.13}
\|f|\dfh\|_2\ls&&\inf_\oz\lf\|\lf\{\sum_{\ell\in\zz}2^{\ell sq}
\lf[\sum_{k=0}^\fz 2^{-kNr}2^{(k+\ell)n}\r.\r.\r.\noz\\
&&\times\lf.\lf.\lf.\int_\rn\frac{\left[
\int_1^2|(\Phi_{k+\ell})_t *f(y)|^q[\omega(\cdot,2^{-\ell}t)
]^{-q}\,\frac{dt}t\r]^{\frac rq}}{
(1+2^\ell|\cdot-y|)^{ar}}\,dy\r]^{\frac qr}
\r\}^{1/q}\r\|_\lp.\qquad\hs
\end{eqnarray}

Choosing $\delta\in(0,a-n/r)$, by H\"older's inequality,
\eqref{3.12} and \eqref{3.13}, we know
that
\begin{eqnarray*}
&&\|f|\dfh\|^v_2\\
&&\hs\ls\sum_{k=0}^\fz2^{-k(N-\dz)v+knv/r}\inf_\oz\lf\|\lf
\{\sum_{\ell\in\zz}2^{\ell (s+n/r)q}\lf(\sum_{i=0}^\fz2^{-iar}
\r.\r.\r.\\
&&\hs\hs\lf.\lf.\lf.\times\int_{|\cdot-y|\sim2^{i-\ell}}\left[
\int_1^2|(\Phi_{k+\ell})_t *f(y)|^q[\omega(\cdot,2^{-\ell}t)
]^{-q}\,\frac{dt}t\r]^{r/q}\,dy\r)^{q/r}\r\}^{1/q}\r\|_\lp^v\noz\\
&&\hs\ls\sum_{k=0}^\fz2^{-k(N-\dz)v+knv/r}\sum_{i=0}^\fz2^{-i(a-\delta)v}\inf_\oz
\lf\|\lf\{\sum_{\ell\in\zz}2^{\ell (s+n/r)q}\r.\r.\\
&&\hs\hs\lf.\lf.\times\lf(
\int_{|\cdot-y|\sim2^{i-\ell}}\lf[
\int_1^2|(\Phi_{k+\ell})_t *f(y)|^q[\omega(\cdot,2^{-\ell}t)
]^{-q}\,\frac{dt}t\r]^{r/q}\,dy\r)^{q/r}\r\}^{1/q}\r\|_\lp^v,
\end{eqnarray*}
where $v$ is as in Remark \ref{r3.2}(ii) and
\begin{equation}\label{3.x1}
|\cdot-y|\sim2^{i-\ell}\quad{\rm means\ that }\quad
2^{i-\ell}\le|\cdot-y|<2^{i-\ell+1}.
\end{equation}
For
$(x,s)\in\rn\times(0,\fz)$, let
$$\oz_{k,i}(x,s)\ev2^{-(i+k)n\tau}\sup\{\wz\oz(y,t):\ |x-y|<2^{i+1}t,\
2^{-k-1}\le s/t\le2^{k+1}\}.$$
Then by Lemma \ref{l3.3},  $\oz_{k,i}(x,s)$ also satisfies
\eqref{2.2} modulo a positive constant. Thus, by
choosing $\oz\ev\oz_{k,i}$ and
the Fefferman-Stein vector-valued inequality \eqref{2.3}, we have
\begin{eqnarray*}
&&\|f|\dfh\|_2^v\\
&&\hs\ls\sum_{k=0}^\fz\sum_{i=0}^\fz2^{-k(N-\dz)v+knv/r}2^{-i(a-\delta)v}
\lf\|\lf\{\sum_{\ell\in\zz}2^{\ell (s+n/r)q}\r.\r.\\
&&\hs\hs\lf.\lf.\lf.\times\lf(
\int_{|\cdot-y|\sim2^{i-\ell}}\r.\left[
\int_1^2|(\Phi_{k+\ell})_t *f(y)|^q[\omega_{k,i}(\cdot,2^{-\ell}t)
]^{-q}\,\frac{dt}t\r]^{r/q}\,dy\r)^{q/r}\r\}^{1/q}\r\|_\lp^v\noz\\
&&\hs\ls\sum_{k=0}^\fz\sum_{i=0}^\fz2^{-k(N-\dz)v+knv/r}2^{-i(a-\delta)v+(k+i)n\tau
v} 2^{inv/r}
\lf\|\lf\{\sum_{\ell\in\zz}2^{\ell sq}
\r.\r.\\
&&\hs\hs\lf.\lf.\times\lf[M\left(\left[
\int_1^2|(\Phi_{k+\ell})_t *f|^q[\wz{\omega}(\cdot,2^{-k-\ell}t)
]^{-q}\,\frac{dt}t\r]^{r/q}\r)\r]^{q/r}\r\}^{1/q}\r\|_\lp^v\noz\\
&&\hs\ls\|f|\dfh\|_1^v,\noz
\end{eqnarray*}
namely, $\|f|\dfh\|_2\ls\|f|\dfh\|_1.$

Now let us compare the characterization by the continuous Peetre maximal function
of the local means and the characterization by the tent space.
By the definitions of $\|f|\dfh\|_2$ and $\|f|\dfh\|_3$, we immediately obtain
$$\|f|\dfh\|_3\ls\|f|\dfh\|_2.$$

Next we prove that $\|f|\dfh\|_2\ls\|f|\dfh\|_3.$
By \eqref{3.9}, we see that
\begin{eqnarray*}
&&\|f|\dfh\|_2\\
&&\hs\ls\inf_\oz\lf\|\lf\{\sum_{\ell\in\zz} 2^{\ell sq}
\lf(\sum_{k=0}^\fz2^{-kNr+2(k+\ell)n}\sum_{i=0}^\fz 2^{-iar}
\int_{|\cdot-y|\sim 2^{i-\ell}}\r.\r.\r.\noz\\
&&\hs\hs\times\lf.\lf.\lf.
\left[\int_1^2\int_{|z|<2^{-(k+\ell)}t}
|(\Phi_{k+\ell})_t*f(y+z)|^q[\oz(\cdot,2^{-\ell}t)]^{-q}
\,dz\,\frac{dt}t\r]^{\frac rq}\,
dy\r)^{\frac qr}\r\}^{\frac 1q}\r\|_\lp,\noz
\end{eqnarray*}
where $|\cdot-y|\sim2^{i-\ell}$ means the same as in \eqref{3.x1}.
Let
$$\oz_{i,k}(x,t)\ev
2^{-(i+k+4)n\tau}\sup\lf\{\wz\oz(y,s):\ |x-y|\le2^{i+1}t,\
2^{-k-2}\le t/s\le2^{k+2}\r\}.$$
Applying Lemma \ref{l3.3} with $\beta=2^{k+i+4}$, we know that
$\oz_{i,k}$ satisfies \eqref{2.2} modulo a positive constant.
Choosing $\delta\in(0,a-n/r)$ and applying H\"older's inequality, by \eqref{3.12},
we then conclude that
\begin{eqnarray*}
&&\|f|\dfh\|_2^v\\
&&\hs\ls\sum_{k=0}^\fz2^{-k(N-\dz)v+2knv/r}\sum_{i=0}^\fz2^{-i(a-\delta)v}
\inf_{\oz}\lf\|\lf\{\sum_{\ell\in\zz} 2^{\ell sq}
\lf(2^{2\ell n}\int_{|\cdot-y|\sim 2^{i-\ell}}\r.\r.\r.\noz\\
&&\hs\hs\times\lf.\lf.\lf.
\left[\int_1^2\int_{|z|<2^{-(k+\ell)}t}
|(\Phi_{k+\ell})_t*f(y+z)|^q[\oz(\cdot,2^{-\ell}t)]^{-q}\,dz\,\frac{dt}t\r]^{r/q}\,
dy\r)^{q/r}\r\}^{1/q}\r\|_\lp^v\noz\\
&&\hs\ls\sum_{k=0}^\fz2^{-k(N-\dz)v+knv/r}
\sum_{i=0}^\fz2^{-i(a-\dz)v}2^{(i+k)n\tau v}2^{inv/r}
\lf\|\lf\{\sum_{\ell\in\zz}
\r.\r.\noz\\
&&\hs\hs\times \lf.\lf.
\lf[M\lf(\int_{2^{-(k+\ell)}}^{2^{-(k+\ell)+1}}
\int_{|z|< t}t^{-sq} |\Phi_{t}*f(\cdot+z)|^q
[\wz\oz(\cdot+z,t)]^{-q}\,dz\,\frac{dt}{t^{n+1}}\r)
\r]^{q/r} \r\}^{1/q}\r\|_\lp^v\noz,
\end{eqnarray*}
where $v$ is as in Remark \ref{r3.2}(ii).
From this and the Fefferman-Stein vector-valued maximal inequality \eqref{2.3}, we
further deduce that
$\|f|\dfh\|_2\ls\|f|\dfh\|_3$.

Finally we show that
$\{f\in\cs'_R(\rn):\ \|f|\dfh\|_i<\fz\}$ for $i\in\{1,\cdots,5\}$
characterizes the space $\dfh$.
Indeed, similar to Step 2 of the proof of
Theorem \ref{t3.1}, with Lemma \ref{l2.1} replaced by Lemma \ref{l2.2},
we know that
$\|f|\dfh\|_i$, $i\in\{1,\cdots,5\}$, are independent of the choice
of $\Phi$ satisfying \eqref{3.2}, which further implies that
$\|f|\dfh\|_i$ are equivalent to $\|f\|_{\dfh}$.

By $R\ge L$ and Lemma \ref{l3.4},
we know that if $f\in\dfh$, then $f\in\cs'_R(\rn)$, which,
together with the fact
that $\cs'_R(\rn)\subset \cs'_\fz(\rn)$, implies the desired conclusion,
and hence completes the proof of Theorem \ref{t3.3}.
\end{proof}

Similar to Theorem \ref{t3.3}, for the space $\dbh$,
we also have the following characterizations. We omit the details here.

\begin{theorem}\label{t3.4}
Let $s\in\rr$, $p\in(1,\fz)$,
$q\in[1,\fz)$, $\tau\in[0, 1/{(p\vee q)'}]$, $R\in\zz_+\cup\{-1\}$ and
$a\in(n(1/p+\tau), \fz)$ such that $s+n\tau<R+1$ and $\Phi$ be as in \eqref{3.2}.
Then the space $\dbh$ is
characterized by
$$\dbh=\{f\in\cs'_R(\rn):\ \|f|\dbh\|_i<\fz\},\quad i\in\{1,\cdots,4\},$$
where
\begin{eqnarray*}
\|f|\dbh\|_1\ev
\inf_\oz\lf\{\int_0^\fz t^{-sq}\|\Phi_t*f(\cdot)[\oz
(\cdot,t)]^{-1}\|_\lp^q\dt\r\}^{1/q},\noz
\end{eqnarray*}
\begin{eqnarray*}
\|f|\dbh\|_2\ev
\inf_\oz\lf\{\int_0^\fz t^{-sq}\|(\Phi_t^*f)_a(\cdot)[\oz
(\cdot,t)]^{-1}\|_\lp^q\dt\r\}^{1/q},\noz
\end{eqnarray*}
\begin{eqnarray*}
\|f|\dbh\|_3\ev
\|\{2^{ks}\Phi_k*f\}_{k\in\zz}\|_{\wz{\ell^q(L^p_\tau(\rn,\zz))}}
\end{eqnarray*}
and
\begin{eqnarray*}
\|f|\dbh\|_4\ev
\|\{2^{ks}(\Phi_k^*f)_a\}_{k\in\zz}\|_{\wz{\ell^q (L^p_\tau(\rn,\zz))}};
\end{eqnarray*}
moreover, when $a\in(2n/p+n\tau, \fz)$,
$$\dbh=\{f\in\cs'_R(\rn):\ \|f|\dbh\|_5<\fz\},$$
where
\begin{eqnarray*}
\|f|\dbh\|_5
\ev\inf_\oz\lf\{\sum_{j\in\zz} 2^{jsq}
\lf\| \lf[\int_\rn\frac{2^{jn}|\Phi_j\ast
f(\cdot+y)|^r}{(1+2^j|y|)^{ar}}
\,dy\r]^{\frac1r}[\oz
(\cdot,t)]^{-1}\r\|^q_\lp\r\}^{\frac1q},\noz
\end{eqnarray*}
where $\omega$ runs over all nonnegative Borel measurable
functions on $\R_+^{n+1}$
satisfying \eqref{2.2} and $r\in (0, p)$ satisfying that $(a-n\tau)r>2n$.
Furthermore, all $\|\cdot|\dbh\|_i$, $i\in\{1,\cdots,5\}$, are equivalent
(quasi-)norm in $\dbh$.
\end{theorem}

With slight modifications for the proofs of Theorems \ref{t3.3} and \ref{t3.4},
we obtain the following observation.

\begin{remark}\label{r3.3}
Theorems \ref{t3.3} and \ref{t3.4} are still
true if we replace $(\Phi_t^*f)_a$ in Theorems \ref{t3.3} and \ref{t3.4}
by $(\widetilde{\Phi}_t^*f)_a$, which is defined in
Remark \ref{r3.1}. The counterparts of Corollaries \ref{c3.1} and
\ref{c3.2} keep valid also in the framework of Theorems \ref{t3.3} and
\ref{t3.4}.
\end{remark}

\section{The Coorbit Space Theory for Quasi-Banach Spaces\label{Classcoo}}

In this section we recall some basic results from \cite{ra05-3}. This theory
is the continuation to quasi-Banach spaces of the classical coorbit space theory
developed by Feichtinger and Gr\"ochenig \cite{FeGr86,FeGr89a,FeGr89b,Gr91} in
the eighties. The
ingredients are a \emph{locally compact group}
$\mathcal{G}$ with \emph{identity} $e$, a \emph{Hilbert space} $\mathcal{H}$ and an
\emph{irreducible,
unitary and continuous representation} $\pi:\mathcal{G} \to
\mathcal{L}(\mathcal{H})$, which is at least integrable. One can associate a
(quasi-)Banach space $\CoY$ to any
solid, translation-invariant (quasi-)Banach space $Y$ of functions on the group
$\mathcal{G}$. This approach provides a powerful discretization machinery for
the space $\CoY$, namely, a universal approach to
atomic decompositions and Banach frames.
In connection with smoothness spaces of Besov-Triebel-Lizorkin-Hausdorff type,
we measure smoothness of a
function in decay properties of the continuous wavelet transform $W_g f$; see
\eqref{wg} below. We show in Section \ref{ax+b} that homogeneous Besov and
Triebel-Lizorkin type spaces as well as homogeneous Besov-Hausdorff and
Triebel-Lizorkin-Hausdorff spaces are represented as coorbits of properly
chosen spaces $Y$ on the $ax+b$-group
$\mathcal{G}$.

\subsection{Function Spaces on $\mathcal{G}$\label{FSgroup}}

Integration on $\mathcal{G}$ is always
with respect to the left Haar measure $d\mu(x)$. The Haar module on
$\mathcal{G}$ is denoted by $\Delta$.
We define further $L_x F(y) \ev F(x^{-1}y)$ and
$R_x F(y) \ev F(yx)$, for all $x,y \in \mathcal{G}$, the
\emph{left} and \emph{right translation
operators,} respectively. A \emph{quasi-Banach function space} $Y$
on the group $\mathcal{G}$ is
supposed to have the following properties:

  {\rm (i)} $Y$ contains all characteristic functions $\chi_K$ of
compact subsets $K \subset \mathcal{G}$;

  {\rm (ii)} $Y$ is invariant under the left and right translations $L_x$ and $R_x$,
which represent in addition continuous operators on $Y$;

  {\rm (iii)} $Y$ is solid, namely, $H\in Y$ and $|F(x)| \le |H(x)|$
almost everywhere imply that $F\in Y$ and $\|F|Y\| \le \|H|Y\|$.

By the Aoki-Rolewicz theorem, there exists $p\in(0,1]$ such that Y has an
equivalent $p$-norm (see \cite[p.\,239]{ra05-3} for the definition),
moreover, $\|\sum_{j\in\zz}y_j|Y\|^p\ls \sum_{j\in\zz}\|y_j|Y\|^p$
for all $\{y_j\}_{j\in\zz}\subset Y$.

A continuous weight $w$ is called \emph{sub-multiplicative}
if $w(xy) \le w(x) w(y)$ for all $x,y \in \mathcal{G}$.
The \emph{(quasi-)Banach space }$L^p_w(\mathcal{G})$, $p\in(0,\infty]$, of functions
$F$ on the group $\mathcal{G}$ is defined via the \emph{norm}
$\|F|L^p_w(\mathcal{G})\| :=\{\int_{\mathcal{G}}
|F(x)w(x)|^p\,d\mu(x)\}^{1/p}\,,$
where the usual modification is made when $p=\infty$\,. If $w := 1$ then
we simply write $L_w^p(\cg)$ by
$L^p(\mathcal{G})$\,. It is easy to show that these spaces have left and
right translation invariances if $w$ is sub-multiplicative.

\subsection{Wiener-Amalgam Spaces}

We follow the notation in \cite{ra05-3}. Let $B$ be one of
the spaces $L^1(\mathcal{G})$ or $L^{\infty}(\mathcal{G})$. Choose one
relatively compact neighborhood $Q$ of $e \in \mathcal{G}$ and define
the \emph{control function}
$K(F,Q,B)(x) := \|(L_x \chi_{Q})F|B\|$ for all $x\in \mathcal{G}$,
where $F$ is locally contained in $B$,
which means $F\chi_K\st B$ for any compact subset $K$ of $\cg$
and is denoted by $F \in B_{\mathrm{loc}}$. Let now
$Y$ be some solid quasi-Banach space of functions on $\mathcal{G}$ containing
the characteristic function of any compact subset of $\mathcal{G}$. The
\emph{Wiener-Amalgam space} $W(B,Y)$ is then defined by
$$
        W(B,Y):= W(B,Y,Q) :=\{F \in B_{\mathrm{loc}}~:~K(F,Q,B)\in Y\}
$$
with (quasi-)norm
$\|F|W(B,Y,Q)\|:=\|K(F,Q,B)|Y\|,$
which has an equivalent $p$-norm with $p$ being
the exponent of the quasi-norm of $Y$.

The following lemma is essentially Theorem 3.1 in the preprint
version of \cite{ra05-3}.

\begin{lemma} The following statements are equivalent:

 {\rm (i)} The spaces $W(L^{\infty}(\cg),Y) = W(L^{\infty}(\cg),Y,Q)$
 is independent of the
 choice of the neighborhood $Q$ of $e$ (with equivalent norms for
 different choices).

 {\rm (ii)} The space $W (L^{\infty}(\cg) , Y, Q)$ is right translation invariant.
\end{lemma}

\subsection{Sequence Spaces\label{sequspace}}

\begin{definition} Let $X \ev \{x_i\}_{i \in I}$ be some discrete set of
points in $\mathcal{G}$ and $V$ be a relatively compact neighborhood
of $e\in \mathcal{G}$\,.

{\rm (i)} $X$ is called \emph{$V$-dense} if $\mathcal{G} = \cup_{i\in I}
x_iV$.

{\rm (ii)} $X$ is called \emph{relatively separated} if for all compact sets
$K\subset \mathcal{G}$, there exists a positive constant $C_K$ such that
$\sup_{j\in I} \sharp\{i\in I~:~x_iK \cap x_jK \neq
\emptyset\} \le C_K\,,$
where and in what follows, $\sharp E$ denotes the \emph{cardinality} of the
set $E$.

{\rm (iii)} $X$ is called \emph{$V$-well-spread} (or simply \emph{well-spread}) if
it is both relatively separated and $V$-dense.
\end{definition}

\begin{definition}\label{sequsp} For a family $X\ev\{x_i\}_{i\in I}$ which is
$V$-well-spread
with respect to a relatively compact neighborhood $V$ of $e\in
\mathcal{G}$, the \emph{sequence space $Y^{\sharp}$} associated to
$Y$ is defined as
\begin{equation}\nonumber
    Y^{\sharp} \ev \lf\{\{\lambda_i\}_{i\in I}~:~
    \|\{\lambda_i\}_{i\in I}|Y^{\sharp}\| \ev \lf\|\sum\limits_{i\in I}
    |\lambda_i|\chi_{x_iV}|Y\r\|<\infty\,\r\}\,.
\end{equation}
\end{definition}

\subsection{Coorbit Spaces}\label{s4.4}
Having a Hilbert space $\mathcal{H}$ and an integrable, irreducible, unitary and
continuous
representation $\pi:\mathcal{G} \to \mathcal{L}(\mathcal{H})$, then the \emph{general
voice transform} of $f\in \mathcal{H}$ with respect to a fixed atom $g\in\mathcal{H}$
is defined as the \emph{function $V_gf$} on the group $\mathcal{G}$
given by
\begin{equation}\label{voice}
    V_gf(x) := \langle \pi(x) g,f\rangle\,,\quad x\in\cg,
\end{equation}
where the bracket denotes the \emph{inner product} in $\mathcal{H}$\,.

\begin{definition}\label{H1w}  For a sub-multiplicative weight $w(\cdot)\geq 1$
on $\mathcal{G}$, the \emph{space
$A_w \subset \mathcal{H}$} of admissible
vectors is defined by
$A_w := \{g\in \mathcal{H}~:~V_g g \in L^1_w(\mathcal{G})\}\,.$
If $A_w \neq \{0\}$ and $g\in A_w\setminus\{0\}$, further define
$\mathcal{H}^1_w := \{f\in \mathcal{H}~:~\|f|\mathcal{H}^1_w\| = \|V_g
f|L^1_w(\mathcal{G})\| <\infty\}\,.$
Finally, denote by $(\mathcal{H}^1_w)^{\sim}$ the \emph{canonical anti-dual} of
$\mathcal{H}^1_w$, namely, the space of all conjugate linear functionals on
$\mathcal{H}^1_w$\,.
\end{definition}
We see immediately that $A_w \subset  \mathcal{H}^1_w \subset \mathcal{H}$. The
voice transform \eqref{voice} is extended to
$\mathcal{H}_w\times (\mathcal{H}^1_w)^{\sim}$ by the usual dual pairing.
Notice that the space $\mathcal{H}^1_w$ is considered as the \emph{space of test
functions} and $(\mathcal{H}^1_w)^{\sim}$ as \emph{reservoir} or
the \emph{space of distributions}.

To treat also (quasi-)Banach spaces, we need the modification \cite{ra05-3} of
the classical coorbit space theory. For
$p\in(0,1]$ and
a sub-multiplicative weight $w$, let us define the
following \emph{set of analyzing vectors}
\begin{equation}\label{4.2}
  B_w^p := \{g\in \mathcal{H}~:~V_{g}g \in W(L^{\infty}(\cg),\,L^p_w(\cg))\}\,.
\end{equation}
In the sequel, we admit only those parameters $p$ and $w$
such that $B_w^p$ contains non-zero functions.

Before defining the coorbit space $\CoY$, we need to define the \emph{weight $w_Y$}
which depends on the space $Y$ on $\mathcal{G}$ satisfying
{\rm (i)} - {\rm (iii)} in Subsection \ref{FSgroup}.
We define further
\begin{equation}\label{wy}
    w_Y(x) := \max\{\|L_{x}\|, \|L_{x^{-1}}\|, \|R_{x}\|,
\Delta(x^{-1})\|R_{x^{-1}}\| \},\quad x\in \mathcal{G}\,,
\end{equation}
where the operator norms are considered from $W(L^{\infty}(\cg),\,Y)$ to
$W(L^{\infty}(\cg),\,Y)$\,. Finally, we put
\begin{equation}\label{4.4}
     B(Y):= B^p_{w_Y} \cap A_{w_Y}\,,
\end{equation}
where $p\in(0,1]$ is such that $Y$ has a $p$-norm.

\begin{definition}\label{defcoob} Let $Y$ be a (quasi-)Banach space,
with an equivalent $p$-norm, on
$\mathcal{G}$ satisfying {\rm (i)}-{\rm (iii)} in Subsection \ref{FSgroup}
and let the weight $w_Y(x)$ be given by \eqref{wy}. Let further
$g\in B(Y)$. The \emph{space $\CoY$}, called the \emph{coorbit space} of $Y$,
is defined by
\begin{equation}\label{defcoorbit}
  \CoY \ev \{f\in (\mathcal{H}_w^1)^{\sim}~:~V_g f \in
W(L^{\infty}(\cg),\,Y)\}
  \end{equation}
  with $\|f|\CoY\| \ev \|V_g f|W(L^{\infty}(\cg),\,Y)\|.$
\end{definition}

The following basic properties are proved for instance in Theorem 4.3 in the
preprint version of \cite{ra05-3}. Notice that Theorem \ref{indep}(i)
is also included in \cite[Theorem 4.4]{ra05-3}
(see also \cite[Theorem 4.5.13]{ra05-6}
for the case of Banach spaces).

\begin{theorem}\label{indep} {\rm (i)} The space $\CoY$ is a (quasi-)Banach space
independent of the analyzing vector $g \in B(Y)$.

{\rm (ii)} The definition of the space $\CoY$ is independent of the
reservoir in the following sense: Assume that $S \subset \mathcal{H}^1_w$
is a non-trivial locally convex vector space which is invariant under
$\pi$. Assume further that there exists a non-zero
      vector $g \in S\cap B(Y)$ for which the reproducing formula
      $V_g f = V_g g \ast V_g f$ holds true for all $f\in S^{\sim}$. Then
      $$\CoY = \{f\in (\mathcal{H}_w^1)^{\sim}~:~V_g f \in Y\} = \{f\in
        S^{\sim}~:~V_g f \in W(L^{\infty}(\cg),\,Y)\}\,.$$
\end{theorem}

\subsection{Discretizations}
\label{groupdisc}

This subsection states the main abstract discretization result. We are interested
in atoms of type $\{\pi(x_i)g\}_{i\in I}$, where $\{x_i\}_{i\in I} \subset
\mathcal{G}$ represents a discrete subset, whereas $g \in B(Y)$ denotes a
fixed admissible analyzing vector.

The following powerful result goes back to Gr\"ochenig \cite{Gr88}
and was further extended by Rauhut \cite{ra05-3}, Rauhut and Ullrich
\cite[Theorem 3.14]{RaUl10}, and recently by Sch\"afer \cite[Theorem 6.8]{sch}
to the quasi-Banach situation.

\begin{theorem}\label{wbases} Suppose that the functions $g_r, \gamma_r \in
B(Y)$, $r\in\{1,...,n\}$. Let $X \ev \{x_i\}_{i\in I}$ be a well-spread set such
that
\begin{equation}\label{waveletexp}
    f = \sum\limits_{r=1}^n \sum\limits_{i\in I} \langle
    \pi(x_i)\gamma_r, f\rangle \pi(x_i)g_r
\end{equation}
for all $f\in \mathcal{H}$\,. Then the expansion \eqref{waveletexp} extends to all
$f\in \CoY$. Moreover, $f\in (\mathcal{H}_w^1)^{\sim}$ belongs to $\CoY$ if and
only if $\{\langle \pi(x_i)\gamma_r,f\rangle\}_{i\in I}$ belongs to
$Y^{\sharp}$ for each $r\in\{1,...,n\}$\,. For $f \in \CoY$, the
expansion \eqref{waveletexp} always converges unconditionally in the
weak $\ast$-topology induced by $(\mathcal{H}_w^1)^{\sim}$. If, in addition,
the finite sequences are dense in $Y^{\sharp}$, then \eqref{waveletexp}
converges unconditionally in the (quasi-)norm of $\CoY$.
\end{theorem}

\section{Coorbit Characterizations}\label{ax+b}

In this section, we always let
$\mathcal{G} := \mathbb{R}^n \rtimes \mathbb{R}_+^*$ be the
$n$-dimensional $ax+b$-group.
In Subsection \ref{s5.1}, we first introduce four classes of Peetre-type
spaces, $\dlt$, $\dpt$, $\dlh$ and $\dph$, on $\cg$,
and show that the left and right translations of $\cg$
are bounded on these Peetre-type spaces in Propositions \ref{p5.1} and \ref{p5.2}.
Combining these boundednesses of translations with the coorbit theory of Rauhut
\cite{ra05-3} (see also Theorem \ref{indep}),
in Theorems \ref{t5.1} and \ref{t5.2} of Subsection \ref{s5.2},
we then prove that the spaces $\dbt$, $\dft$, $\dbh$
and $\dfh$ are, respectively, the coorbit spaces of the Peetre-type spaces
$\dot{L}^{s+n/2-n/q,\tau}_{p,q,a}(\cg)$, $\dot{P}^{s+n/2-n/q,\tau}_{p,q,a}(\cg)$,
$L\dot{H}^{s+n/2-n/q,\tau}_{p,q,a}(\cg)$ and
$P\dot{H}^{s+n/2-n/q,\tau}_{p,q,a}(\cg)$.
Finally, in Subsection \ref{sequhom},
we introduce some sequence spaces corresponding to
$\dlt$, $\dpt$, $\dlh$ and $\dph$, which are used in Section \ref{s6}
to obtain the wavelet characterizations of the spaces
$\dbt$, $\dft$, $\dbh$
and $\dfh$.

\subsection{Peetre-Type Spaces on $\mathcal{G}$\label{s5.1}}

Recall that the \emph{group operation} of $\cg$ is given by
$$  (x,t)(y,s) \ev (x+ty,st)\quad x,\ y \in \rn, \, s,\ t\in(0,\fz).$$
The \emph{left Haar measure} $\mu$ on $\mathcal{G}$ is given by
$d\mu(x,t) \ev dx\,dt/t^{n+1}$ and the \emph{Haar module} is $\Delta(x,t) \ev t^{-n}$.
Giving a function $F$ on $\mathcal{G}$, for any ${\bf y}\ev(y,r)\in\cg,$
the \emph{left and right translations}
$L_{\bf y} \ev L_{(y,r)}$ and $R_{\bf y} \ev R_{(y,r)}$ are given by
setting, for all $(x,t)\in\cg$,
$$L_{(y,r)}F(x,t) = F((y,r)^{-1}(x,t)) = F\lf(\frac{x-y}{r},\frac{t}{r}\r)$$
and
$R_{(y,r)}F(x,t) = F((x,t)(y,r)) = F(x+ty,rt)\,.$

\begin{definition}\label{d5.1}
Let $s\in\rr$, $\tau\in[0,\fz)$ and $q\in(0,\,\fz]$.

{\rm (i)} The {\it  space $\dlt$} with $p\in(0,\fz]$ is defined to be the
set of all functions $F$ on $\mathcal{G}$ such that
\begin{eqnarray*}
&&\|F|\dlt\|\\
&&\hs\ev\sup_{P\in\cq}\frac1{|P|^\tau}
\lf\{\int_0^{\ell(P)}t^{-sq}\lf[\int_P\lf(
\sup_{\substack{y\in\rn\\t/2\le r \le
 t}}\frac{|F(x+y,r)|}{(1+|y|/r)^a}
\r)^p\,dx\r]^{q/p}\,\frac{dt}{t^{n+1}}\r\}^{1/q}<\fz.\noz
\end{eqnarray*}

{\rm (ii)} The {\it  space $\dpt$} with $p\in (0,\fz)$ is defined to be the
set of all functions $F$ on $\mathcal{G}$ such that
\begin{eqnarray*}
&&\|F|\dpt\|\\
&&\hs\ev\sup_{P\in\cq}\frac1{|P|^\tau}
\lf\{\int_P\lf[\int_0^{\ell(P)}t^{-sq}\lf(
\sup_{\substack{y\in\rn\\t/2\le r \le
 t}}\frac{|F(x+y,r)|}{(1+|y|/r)^a}
\r)^q\,\frac{dt}{t^{n+1}}\r]^{p/q}\,dx\r\}^{1/p}<\fz.\noz
\end{eqnarray*}
\end{definition}

If we change dyadic cubes $\mathcal{Q}$ in Definition \ref{d5.2}
into the set of all cubes in $\rn$ whose sides parallel
to the coordinate axes, we obtain equivalent (quasi-)norms.

Next we show that the left
and right translations are bounded on $\dlt$ and $\dpt$.
In what follows, for a given (quasi)-Banach space $X$ and a
linear operator $T$ from $X$ to $X$, we denote by $\|T\|_{X\to X}$
the \emph{operator norm} of $T$.

\begin{proposition}\label{p5.1}
Let $s\in\rr$, $\tau\in[0,\fz)$ and $q\in(0,\,\fz]$. Then the left
and right translations are bounded on $\dlt$ and $\dpt$.
Moreover, the following estimates hold:
\begin{equation}\label{5.1}
\|R(z,r)\|_{\dlt\to\dlt}\le
Cr^{s+n/q}(1\vee r^{-a})(1\vee r^{n\tau})(1+|z|)^a,
\end{equation}
\begin{equation}\label{5.2}
\|R(z,r)\|_{\dpt\to\dpt}\le
Cr^{s+n/q}(1\vee r^{-a})(1\vee r^{n\tau})(1+|z|)^a,
\end{equation}
\begin{equation}\label{5.3}
\|L(z,r)\|_{\dlt\to\dlt}\le Cr^{n(1/p-1/q)-s-n\tau}
\end{equation}
and
\begin{equation}\label{5.4}
\|L(z,r)\|_{\dpt\to\dpt}\le Cr^{n(1/p-1/q)-s-n\tau}.
\end{equation}
\end{proposition}

\begin{proof}
By similarity, we only give the proofs for \eqref{5.1} and \eqref{5.4}.
For \eqref{5.1}, notice that
\begin{equation}\label{5.5}
\sup_{\gfz{y\in\rn}{t/2\le\theta\le t}}
\frac{|F(x+y+\theta z,\theta r)|}{(1+|y|/\theta)^a}
\le
(1\vee r^{-a})(1+|z|)^a
\sup_{\gfz{y\in\rn}{tr/2\le\theta\le tr}}\frac{|F(x+y,\theta)|}{(1+|y|/\theta)^a}.
\end{equation}
By \eqref{5.5} and changing variables, we then see that
\begin{eqnarray*}
&&\|R(z,r)F\|_{\dlt}\\
&&\hs\le
(1\vee r^{-a})(1+|z|)^a\\
&&\hs\hs\times\sup_{P\in\cq}\frac1{|P|^\tau}
\lf\{\int_0^{\ell(P)}t^{-sq}\lf[\int_P\lf(
\sup_{\substack{y\in\rn\\tr/2\le \theta \le
tr}}\frac{|F(x+y,\theta)|}{(1+|y|/\theta)^a}
\r)^p\,dx\r]^{q/p}\,\frac{dt}{t^{n+1}}\r\}^{1/q}\\
&&\hs\ls
r^{s+n/q}(1\vee r^{-a})(1\vee r^{n\tau})(1+|z|)^a\|F|\dlt\|,
\end{eqnarray*}
which implies \eqref{5.1}.

Next we prove \eqref{5.4}. Notice that
\begin{eqnarray*}
&&\|L_{(z,r)}F|\dpt\|\\
&&\hs=\sup_{P\in\cq}\frac1{|P|^\tau}
\lf\{\int_{P-z}\lf[\int_0^{\ell(P)}t^{-sq}\lf(
\sup_{\gfz{y\in\rn}{t/2\le \theta\le t}}
\frac{|F(\frac{x+y}r,\,\frac \theta r)|}{(1+\frac{|y|}\theta)^a}
\r)^q\,\frac{dt}{t^{n+1}}\r]^{p/q}\,dx\r\}^{1/p},
\end{eqnarray*}
where and in what follows, $P-z\ev\{y-z:\ y\in P\}$.
By changing variables, we know that
\begin{eqnarray*}
&&\|L_{(z,r)}F|\dpt\|\\
&&\hs=\sup_{P\in\cq}\frac{r^{n(1/p-1/q)-s}}{|P|^\tau}
\lf\{\int_{\frac{P-z}r}\lf[\int_0^{\ell(P)/r}t^{-sq}\lf(
\sup_{\gfz{y\in\rn}{t/2\le \theta\le t}}
\frac{|F({x+y},\,\theta)|}{(1+\frac{|y|}\theta)^a}
\r)^q\,\frac{dt}{t^{n+1}}\r]^{p/q}\,dx\r\}^{1/p}\\
&&\hs=r^{n(1/p-1/q)-s-n\tau}\sup_{P\in\cq}\frac{1}{|\frac{P-z}{r}|^\tau}
\\
&&\hs\hs\times \lf\{\int_{\frac{P-z}r}
\lf[\int_0^{\ell(\frac{P-z}{r})}t^{-sq}\lf(
\sup_{\gfz{y\in\rn}{t/2\le \theta\le t}}
\frac{|F({x+y},\,\theta)|}{(1+\frac{|y|}\theta)^a}
\r)^q\,\frac{dt}{t^{n+1}}\r]^{p/q}\,dx\r\}^{1/p}\\
&&\hs\ls r^{n(1/p-1/q)-s-n\tau}\|F|\dpt\|,
\end{eqnarray*}
where $\frac{P-z}r\ev\{\frac{y-z}r:\ y\in P\}$.
Thus, $\|L(z,r)\|_{\dpt\to\dpt}\ls r^{n(1/p-1/q)-s-n\tau},$
which completes the proof of Proposition \ref{p5.1}.
\end{proof}

\begin{definition}\label{d5.2}
Let $s\in\rr$, $p\in(1,\,\fz)$ and $\oz$ be a
nonnegative function satisfying \eqref{2.2}.

{\em{\rm (i)}} The {\it space $\dlh$} with $q\in[1,\fz)$
and $\tau\in[0, 1/{(p\vee q)'}]$ is defined to be the
set of all functions $F$ on $\mathcal{G}$ such that
\begin{eqnarray*}
&&\|F|\dlh\|\\
&&\hs\ev\inf_{\oz}\lf\{ \int_0^\fz t^{-sq}\lf[\int_\rn\lf(
\sup_{\substack{y\in\rn\\t/2\le r \le
 t}}\frac{|F(x+y,r)|}{(1+|y|/r)^a}\r)^p[\oz(x,t)]^{-p}
\,dx\r]^{q/p}\,\frac{dt}{t^{n+1}}\r\}^{1/q}<\fz,\noz
\end{eqnarray*}
where the infimum is taken over all
nonnegative Borel measurable functions $\omega$ on $\R_+^{n+1}$ satisfying \eqref{2.2}.

{\rm (ii)} The {\it space $\dph$} with $q\in(1,\fz)$ and
$\tau\in[0, 1/{(p\vee q)'}]$ is defined to be the
set of all functions $F$ on $\mathcal{G}$ such that
\begin{eqnarray*}
&&\|F|\dph\|\\
&&\hs\ev\inf_{\oz}\lf\{ \int_\rn \lf[\int_0^\fz t^{-sq}\lf(
\sup_{\substack{y\in\rn\\t/2\le r \le
 t}}\frac{|F(x+y,r)|}{(1+|y|/r)^a}\r)^q[\oz(x,t)]^{-q}
\frac{dt}{t^{n+1}}\r]^{p/q}\,\,dx\r\}^{1/p}<\fz,\noz
\end{eqnarray*}
where the infimum is taken over all
nonnegative Borel measurable functions $\omega$ on $\R_+^{n+1}$ satisfying \eqref{2.2}.
\end{definition}

Similar to Remark \ref{r3.2},
we have the following conclusions.

\begin{remark}\label{r5.1}\rm
{\rm (i)} The spaces $\dlh$ and $\dph$ are
quasi-Banach spaces. Indeed, for any $F_1,\,F_2\in \dlh$,
$$\|F_1+F_2\|_{\dlh}\le 2^{(p\vee q)'}\lf[\|F_1\|_{\dlh}+\|F_2\|_{\dlh}\r].$$
An inequality similar to this is also true for $\dph$.

{\rm (ii)} By the Aoki-Rolewicz theorem, we know that when
$v:=1/{(1+(p\vee q)')}$, then
$$\lf\|\sum_{j\in\zz} F_j\r\|^v_{\dlh}
\ls \sum_{j\in\zz}\lf\|F_j\r\|^v_{\dlh}\quad \mathrm{for\ all}\
\{F_j\}_{j\in\zz}\subset \dlh,$$
and
$$\lf\|\sum_{j\in\zz} F_j\r\|^v_{\dph}
\ls \sum_{j\in\zz}\lf\|F_j\r\|^v_{\dph}\quad \mathrm{for\ all}\
\{F_j\}_{j\in\zz}\subset \dph.$$
\end{remark}

Next we show that the left
and the right translations are bounded on $\dlh$ and $\dph$.
To this end, we need the following dilation property and the
translation invariance property of Hausdorff capacities.

\begin{lemma}\label{l5.1}
Let $d\in(0,n]$. Then, for all $\theta \in(0,\fz)$ and $E\st\rn$,
$H^d(\theta E)=\theta^dH^d(E)$, where $\theta E\ev\{\theta x:\ x\in E\}$.
\end{lemma}

\begin{proof}
For all $\ez>0$, there exist balls $\{B_j\}_{j\in\nn}$ such that
$E\st\cup_{j}B_j$ and
$$\sum_jr_{B_j}^d-\ez\le H^d(E)\le\sum_jr_{B_j}^d.$$
Then $\theta E\st\cup_{j}(\theta B_j)$ and $\theta B_j$ is a ball. Thus,
$$H^d(\theta E)\le\sum_j(\theta r_{B_j})^d
=\theta^d\sum_jr_{B_j}^d\le \theta^dH^d(E)+\theta^d\ez,$$
which implies that $H^d(\theta E)\le \theta^dH^d(E)$.

Replacing $\theta$ and $E$, respectively, by $\theta^{-1}$ and
$\theta E$ in the above argument,
we immediately obtain the converse inequality, which completes the proof
of Lemma \ref{l5.1}.
\end{proof}

We also have the following observation. Since the proof is trivial, we
omit the details.

\begin{lemma}\label{l5.2}
For all $d\in(0,n]$ and $z\in\rn$,
$H^d(E)=H^d(E+z),$
where $E+z\ev\{x+z:\ x\in E\}$.
\end{lemma}

\begin{proposition}\label{p5.2}
Let $s,\,p,\,q,$ and $\tau$ be as in Definition \ref{d5.2}. Then the left
and right translations are bounded on $\dlh$ and $\dph$.
Moreover, the following estimates hold:
\begin{equation}\label{5.6}
\|R(z,r)\|_{\dlh\to\dlh}\le
Cr^{s+n/q}(1\vee r^{-a})(1\vee r^{-n\tau})(1+|z|)^a,
\end{equation}
\begin{equation}\label{5.7}
\|R(z,r)\|_{\dph\to\dph}\le
Cr^{s+n/q}(1\vee r^{-a})(1\vee r^{-n\tau})(1+|z|)^a,
\end{equation}
\begin{equation}\label{5.8}
\|L(z,r)\|_{\dlh\to\dlh}\le Cr^{n(1/p-1/q)-s+n\tau}
\end{equation}
and
\begin{equation}\label{5.9}
\|L(z,r)\|_{\dph\to\dph}\le Cr^{n(1/p-1/q)-s+n\tau}.
\end{equation}
\end{proposition}

\begin{proof}
We also only prove \eqref{5.6} and \eqref{5.9} by similarity.
For \eqref{5.6}, let $\wz\oz$ be a nonnegative function satisfying \eqref{2.2} and
\begin{eqnarray*}
\lf\{ \int_0^\fz t^{-sq}\lf[\int_\rn\lf(
\sup_{\substack{y\in\rn\\t/2\le r \le
 t}}\frac{|F(x+y,r)|}{(1+|y|/r)^a}\r)^p[\wz\oz(x,t)]^{-p}
\,dx\r]^{q/p}\,\frac{dt}{t^{n+1}}\r\}^{1/q}\ls\|F|\dlh\|.
\end{eqnarray*}
Define $\omega(x,t):= (1\wedge r^{n\tau})\wz\omega(x,tr)$ for all
$(x,t)\in\rr_+^{n+1}$. Then
$$N\omega(x)
=\lf(1\wedge r^{n\tau}\r)\sup_{|x-y|<t/r}\wz\omega(y,t)$$
for all $x\in\rn$.
By Lemma \ref{l3.3},
we know that $\omega$ also satisfies \eqref{2.2}, which, together with \eqref{5.5},
further implies that
\begin{eqnarray*}
\|R(z,r)F|\dlh\|\ls r^{s+n/q}(1\vee r^{-a})(1\vee r^{-n\tau})(1+|z|)^a\|F|\dlh\|.
\end{eqnarray*}
From this, we deduce \eqref{5.6}.

We now prove \eqref{5.9}.
Let $\wz\oz$ be a nonnegative function satisfying \eqref{2.2} and
\begin{eqnarray*}
\lf\{\int_\rn\lf[\int_0^\fz t^{-sq}\lf(
\sup_{y\in\rn}\frac{|F(x+y,t)|}{(1+|y|/t)^a}\r)^q[\wz\oz(x,t)]^{-q}
\frac{dt}{t^{n+1}}\r]^{p/q}\,dx\r\}^{1/p}\ls\|F|\dph\|.
\end{eqnarray*}
Notice that
\begin{eqnarray*}
\|L_{(z,r)}F|\dph\|
&&=r^{-s+n(1/p-1/q)}\inf_\oz\lf\{\int_\rn\lf[\int_0^\fz t^{-sq}
\r.\r.\\
&&\hs\lf.\lf.\times
\lf(\sup_{\gfz{y\in\rn}{t/2\le\theta\le t}}
\frac{|F(x+y,\theta)|}{(1+|y|/\theta)^a}\r)^q[\oz(rx+z,rt)]^{-q}
\frac{dt}{t^{n+1}}\r]^{p/q}\,dx\r\}^{1/p}.
\end{eqnarray*}
Let $\oz(x,t)\ev r^{-n\tau}\wz\oz(\frac{x-z}r,\frac tr)$
for all $(x,t)\in\rr_+^{n+1}$. Then
$\oz(rx+z,rt)=r^{-n\tau}\wz\oz(x,t)$ and
\begin{eqnarray*}
N\oz(x)&&=\sup_{|y-x|<t}\oz(y,t)
=r^{-n\tau}\sup_{|y-x|<t}\wz\oz\lf(\frac{y-z}r,\frac tr\r)
=r^{-n\tau}N\wz\oz\lf(\frac{x-z}r\r).
\end{eqnarray*}
Thus, by Lemmas \ref{l5.1} and \ref{l5.2}, we conclude that
\begin{eqnarray*}
\int_{\R^n} [N\omega(x)]^{(p \vee q)'} \,dH^{n\tau(p\vee q)'}(x)
&&=r^{-n\tau(p\vee q)'}\int_{\R^n}
\lf[N\wz\omega\lf(\frac{x-z}r\r)\r]^{(p \vee q)'} \,dH^{n\tau(p\vee q)'}(x)\\
&&=r^{-n\tau(p\vee q)'}\int_{\R^n}
\lf[N\wz\omega\lf(x\r)\r]^{(p \vee q)'} \,dH^{n\tau(p\vee q)'}(rx+z)\\
&&\le \int_{\R^n} \lf[N\wz\omega\lf(x\r)\r]^{(p
\vee q)'} \,dH^{n\tau(p\vee q)'}(x)\le 1.
\end{eqnarray*}
Thus, $\oz$ satisfies \eqref{2.2}, which further implies that
\begin{eqnarray*}
\|L_{(z,r)}F|\dph\|\ls r^{-s+n(1/p-1/q)+n\tau}\|F\|_\dph.
\end{eqnarray*}
This finishes the proof of Proposition \ref{p5.2}.
\end{proof}

\subsection{Besov-Triebel-Lizorkin-Type Spaces as Coorbits}\label{s5.2}

We put $\mathcal{H}\ev L^2(\R^n)$ and use the
\emph{unitary representation} over $\cg$
$\pi(x,t):= T_x \mathcal{D}_t^{L^2}$ for $x \in \rn$ and $t\in(0,\fz),$
where $T_xf := f(\cdot-x)$ and $\mathcal{D}_t^{L^2}f = t^{-n/2}f(\cdot/t)$.
This representation is unitary, continuous and square integrable on
$\mathcal{H}$ but not irreducible. However, if we
restrict to radial functions $g\in L^2(\R^n)$, then $\mbox{span}
\{\pi(x,t)g~:(x,t)
\in \mathcal{G}\}$ is dense
in $L^2(\R^n)$. The voice transform in this particular
situation is given by the so-called \emph{continuous
wavelet transform} $W_g f$, defined by
\begin{equation}\label{wg}
    W_g f(x,t) := \langle T_x\mathcal{D}^{L^2}_t g, f\rangle,\quad x\in \rn,\
t\in(0,\fz),
\end{equation} where the bracket $\langle \cdot, \cdot \rangle$ denotes the
\emph{inner product} in $L^2(\R^n)$. We write $W_g f$ in terms of a convolution via
\begin{equation}\label{CWTconv}
    W_g f(x,t) = [(\mathcal{D}^{L^2}_t g(-\cdot)) \ast \bar{f}](x) =
t^{d/2}[(\mathcal{D}_t g(-\cdot)) \ast \bar{f}](x)
\end{equation}
for all $x\in \rn$ and $t\in(0,\fz)$.

The next definition helps to control the decay of the continuous wavelet
transform. Both, the definition and the decay result are taken from
\cite[Appendix A]{u10}.

\begin{definition}\label{basedef} Let $L+1\in \zz_+$ and $K\in(0,\fz)$. The
\emph{properties} $(D)$, $(M_L)$ and $(S_K)$ for
a function $f \in L^2(\R^n)$ are, respectively, defined as follows:
\begin{enumerate}
     \item[$(D)$] For every $N\in \N$, there exists a positive constant
    $C(N)$ such that
    $$|f(x)| \le \frac{C(N)}{(1+|x|)^N}\,.$$
    \item[$(M_L)$] All moments up to order $L$ vanish, namely,
    for all $\alpha\in \zz_+^n$ such that $\|\alpha\|_{\ell^1}\le L$\,,
    $$
        \int_{\rn}x^{\alpha}f(x)\,dx = 0.
    $$
    \item[$(S_K)$] The function
    $(1+|\xi|)^{K}|D^{\alpha}\widehat{f}(\xi)|$
    for all $\xi\in\rn$ belongs to $L^1(\R^n)$
    for every multi-index $\alpha\in \zz_+^n$.
\end{enumerate}
\end{definition}

Property $(S_K)$ is rather technical. Suppose that we have a function $f
\in C^{K+n+1}(\R^n)$ for some $K\in \N$ such that $f$ itself and all
its derivatives satisfy $(D)$. The latter holds, for instance, if
$f$ is compactly supported. Then this function satisfies $(S_K)$ by the
elementary properties of the Fourier transform. Conversely, if a
function $g \in L^2(\R^n)$ satisfies $(S_K)$ for some $K\in(0,\fz)$, then we
have $g\in C^{\lfloor K \rfloor}(\R^n)$. However, in case of certain
wavelet functions $\psi$ where the Fourier transform $\widehat\psi$
is given explicitly, we can verify $(S_K)$ directly.

\begin{lemma}\label{help1} Let $L\in \zz_+$, $K\in(0,\fz)$, and $g,f, f_0 \in L^2(\R^n)$.

   {\rm (i)} Let $g$ satisfy $(D)$ and $(M_{L-1})$, and let $f_0$ satisfy $(D)$ and
    $(S_{K})$. Then for every
   $N\in \N$, there exists a positive constant $C(N)$ such that the
   estimate
   $$
       |(W_g f_0)(x,t)| \le C(N) \frac{t^{(L\wedge K)+n/2}}{(1+|x|)^N}
   $$
   holds true for all $x\in \R^n$ and $t\in(0,1)$\,.

   {\rm (ii)} Let $g,f$ satisfy $(D)$, $(M_{L-1})$ and $(S_K)$. For every $N\in
   \N$, there exists a positive constant $C(N)$ such that the
   estimate
   $$
       |(W_g f)(x,t)| \le
C(N)\frac{t^{(L\wedge K)+n/2}}{(1+t)^{2(L\wedge K)+n}}\lf(1+\frac{|x|}{1+t}
\r)^{-N}
   $$
   holds true for all $x\in \R^n$ and $t\in(0,\infty)$\,.
\end{lemma}

Recall that  the abstract definitions of the spaces $\mathcal{H}^1_w, A_w$ and $B^p_w$
are from Definition \ref{H1w}.
The following result is a direct consequence of the definition of
$\mathcal{S}_{\infty}(\rn)$ and Lemma \ref{help1}. It states
that for polynomial weights $w$, the spaces $\mathcal{H}^1_w, A_w$ and $B^p_w$
contain a reasonable set of functions, namely, the Schwartz class
$\mathcal{S}_{\infty}(\R^n)$.

\begin{lemma}\label{emb} If the weight function $w(x,t)\geq 1$ satisfies the
condition that
   \begin{equation}\label{5.12}
        w(x,t) \lesssim (1+|x|)^r(t^s+t^{-s'}),\quad x\in \R^n,\ t\in(0,\fz)\,,
    \end{equation}
    for some nonnegative $r,s,s'$, then
    $\mathcal{S}_{\infty}(\R^n) \subset (\mathcal{H}^1_w \cap A_{w} \cap
     B^p_w)\,.$
\end{lemma}

Relation \eqref{5.12} is a kind of the minimal condition which is needed
in order to define coorbit spaces in a reasonable way. Instead of
$(\mathcal{H}^1_w)^{\sim}$, one may use $\mathcal{S}_{\infty}'(\R^n)$ as
reservoir and a radial $g\in\mathcal{S}_{\infty}(\R^n)$ as
admissible vector; see Definition \ref{H1w}. Considering
\eqref{wy}, we have to restrict to such function spaces $Y$ on $\mathcal{G}$
satisfying (i), (ii), (iii) in
Subsection \ref{FSgroup}, where additionally

\begin{enumerate}
\item[(iv)] $w_Y(x,t) \lesssim (1+|x|)^r(t^s+t^{-s'})$ for all
$x\in \R^n$ and $t\in(0,\fz),$ and some nonnegative $r,s,s'$.
\end{enumerate}

We use the spaces defined in
Subsection \ref{d5.1} as spaces $Y$ and obtain the following main results
of this section.

\begin{theorem}\label{t5.1} Let $s\in \R$, $\tau \in [0,\infty)$ and $q\in
(0,\infty]$.

{\rm (i)} If $p\in (0,\infty)$ and $a \in ({n}/{(p\wedge q)},\infty)$,
then
$
        \dot{F}^{s,\tau}_{p,q}(\R^n) = \Co(\dot{P}^{s+n/2-n/q,\tau}_{p,q,a}(\cg))
$
holds true in the sense of equivalent (quasi-)norms.

{\rm (ii)} If $p \in (0,\infty]$ and $a \in ({n}/{p},\infty)$, then
$
        \dot{B}^{s,\tau}_{p,q}(\R^n) = \Co (\dot{L}^{s+n/2-n/q,\tau}_{p,q,a}(\cg))
$
holds true in the sense of equivalent (quasi-)norms.
\end{theorem}

\begin{proof} Observe that $Q = [-1,1]^n \times [1/2,2] \subset
\mathcal{G}$ is a neighborhood of the
identity $e = (0,1) \in \mathcal{G}$. A simple calculation shows that
$\dot{L}^{s+n/2-n/q,\tau}_{p,q,a}(\cg)=
W(L^{\infty}(\mathcal{G}),\dot{L}^{s+n/2-n/q,\tau}_{p,q,a}(\cg),Q)$
and
$\dot{P}^{s+n/2-n/q,\tau}_{p,q,a}(\cg)=
W(L^{\infty}(\mathcal{G}),\dot{P}^{s+n/2-n/q,\tau}_{p,q,a}(\cg),Q)$
in the notions of Section \ref{Classcoo}.
The results {\rm (i)} and {\rm (ii)} are then a direct consequence of
Definition \ref{defcoob}, \eqref{CWTconv}, Lemma \ref{p5.1}, Theorems
\ref{t3.1} and \ref{t3.2}, by taking Remark \ref{r3.1} into account, and the
abstract independence result in Theorem \ref{indep}. This finishes the
proof of Theorem \ref{t5.1}.
\end{proof}

\begin{theorem}\label{t5.2} Let $s\in \R$ and $p\in (1,\infty)$.

{\rm (i)} If $q \in [1,\infty)$, $\tau \in [0,{1}/{(p\vee q)'}]$ and
$a \in (n({1}/{p}+\tau),\infty)$, then
$B\dot{H}^{s,\tau}_{p,q}(\R^n) =
\Co(L\dot{H}^{s+n/2-n/q,\tau}_{p,q,a}(\cg))$
holds true in the sense of equivalent (quasi-)norms.

{\rm (ii)} If $q\in (1,\infty)$, $\tau \in
[0,{1}/{(p\vee q)'}]$ and $a\in
(n({1}/(p\wedge q)+\tau),\infty)$,
then $F\dot{H}^{s,\tau}_{p,q}(\R^n) = \Co
(P\dot{H}^{s+n/2-n/q,\tau}_{p,q,a}(\cg))$
holds true in the sense of equivalent (quasi-)norms.
\end{theorem}

\begin{proof} Similar as in the proof of Theorem \ref{t5.1}, we have
$$L\dot{H}^{s+n/2-n/q,\tau}_{p,q,a}(\cg)=
W(L^{\infty}(\mathcal{G}),L\dot{H}^{s+n/2-n/q,\tau}_{p,q,a}(\cg),Q)$$
and $P\dot{H}^{s+n/2-n/q,\tau}_{p,q,a} (\cg)=
W(L^{\infty}(\mathcal{G}),P\dot{H}^{s+n/2-n/q,\tau}_{p,q,a}(\cg),Q)\,.$
The results {\rm (i)} and {\rm (ii)} are then a direct consequence of Definition
\ref{defcoob},
\eqref{CWTconv}, Lemma \ref{p5.2}, Theorems
\ref{t3.3} and \ref{t3.4}, by taking Remark \ref{r3.3} into account, and the abstract
independence result in Theorem
\ref{indep}. This finishes the proof of Theorem \ref{t5.2}.
\end{proof}

\subsection{Sequence spaces\label{sequhom}}

In the sequel, we consider a compact neighborhood
of the identity element in $\mathcal{G}$ given by $\mathcal{U} \ev
[-1/2,1/2]^n \times [1/2,1]$. Accordingly, we consider the discrete set of
points
$$
  \{x_{j,k} = ( 2^{-j}k,2^{-j})~:~j\in \zz,\ \ k\in \Z^n\}\,.
$$
This family is $\mathcal{U}$-well-spread. Indeed,
if we write
$$
    U_{jk} \ev [(k_1-1/2)2^{-j},(k_1+1/2)2^{-j}]\times \cdots
\times[(k_d-1/2)2^{-j},(k_d+1/2)2^{-j}]\,,
$$
we then have
$
   x_{j,k}\cdot\mathcal{U} = U_{jk} \times [2^{-(j+1)},2^{-j}]\,.
$
We write $\chi_{j,k} \ev \chi_{U_{jk}}$.
According to Definition \ref{sequsp}, we obtain $Y^\sharp$
in this setting as follows.

\begin{definition}\label{d5.4}
Let $Y$ be a function space on $\mathcal{G}$ as above. Put
   $$
Y^{\sharp}\ev\left\{\{\lambda_{j,k}\}_{j\in
\Z,k\in \Z^n}~:~\|\{\lambda_{j,k}\}_{j,k}|Y^{\sharp} \|
\ev
\lf\|\sum\limits_{j,k}|\lambda_{j,k}|\chi_{j,k} \otimes \chi_{[2^{-(j+1)},2^
{-j}]}|Y\r\|<\infty\right\}\,.
   $$
\end{definition}

For the Peetre-type spaces defined in Subsection \ref{s5.1}, we have
the following equivalent representations.

\begin{proposition}\label{prop:sawano-1}
Let $s\in \R$, $\tau \in [0,\infty)$ and $q\in
(0,\infty]$.

{\rm (i)} If $p \in (0,\infty]$ and $a \in ({n}/{p},\infty)$, then
\begin{eqnarray*}
\lf\|\{\lambda_{j,k}\}_{j,k}|(\dlt)^{\sharp}\r\|
\sim\sup_{P\in\cq}\frac1{|P|^\tau}
\lf\{\sum_{\ell=j_P}^\fz2^{\ell(s+n/q)q}\lf[\int_P\lf(
\sum\limits_{k\in\zz^n}|\lambda_{\ell,k}|\chi_{\ell,k}(x)
\r)^p\,dx\r]^{\frac qp}\r\}^{\frac 1q}\,.
\end{eqnarray*}

{\rm (ii)} If $p\in (0,\infty)$ and $a \in ({n}/{(p\wedge q)},\infty)$,
then
\begin{eqnarray*}
\lf\|\{\lambda_{j,k}\}_{j,k}|(\dpt)^{\sharp}\r\|
\sim\sup_{P\in\cq}\frac1{|P|^\tau}
\lf\{\int_P\lf[
\sum_{\ell=j_P}^\fz\sum_{k\in\zz^n}2^{\ell(s+n/q)q}|
\lambda_{\ell,k}|^q\chi_{\ell,k}(x)\r]^{\frac pq}\,dx\r\}^{\frac 1p}\,.
\end{eqnarray*}

\end{proposition}

To prove this proposition, we need the following
estimate, which is \cite[(4.30)]{u10}.

\begin{lemma}\label{l5.5}
There exist positive constants $C_1$ and $C_2$ such that, for all $\ell\in\zz$, $k\in\zz^n$
and $x\in\rn$,
\begin{equation}\label{5.13}
      \sup\limits_{y\in\rn}\frac{|\chi_{\ell,k}(x+w)|}{(1+2^{\ell}|w|)^{ar}}
\le C_1
      \frac{1}{(1+2^{\ell}|x-k2^{-\ell}|)^{ar}}\le C_2
\lf(\chi_{\ell,k}\ast \frac{2^{\ell
n}}{(1+2^{\ell}|\cdot|)^{ar}}\r)(x).
\end{equation}
\end{lemma}

The following lemma is used
for later consideration, which is \cite[Lemma 3.4]{yy2}.

\begin{lemma}\label{l5.6}
Let $s \in\rr, \, \tau \in [0,\infty)$,
$p,\,q\in (0,\infty]$ and $\lambda>n$.
For a sequence $t\ev\{t_{jk}\}_{j \in \Z, \, k \in \Z^n}$,
define $t^*:=\{t^*_Q\}_{Q \in \cq}$ by setting, for all cubes $Q \in \cq$,
$$
t^*_Q
\ev
\left[\sum_{R \in \cq, \, \ell(R)=\ell(Q)}
\frac{|t_R|^r}{(1+\ell(R)^{-1}|x_R-x_Q|)^\lambda}\right]^{1/r}.
$$
Then
\begin{equation}\label{eq:110322-1}
\|t\|_{\dot{a}^{s,\tau}_{p,q}(\rn)} \sim
\|t^*\|_{\dot{a}^{s,\tau}_{p,q}(\rn)}.
\end{equation}
\end{lemma}

With the above tools, the proof of Proposition \ref{prop:sawano-1}
is similar to that of \cite[Theorem 4.8]{u10}. For completeness,
we give some details here.

\begin{proof}[Proof of Proposition \ref{prop:sawano-1}]
We only prove (i). The proof for (ii) is
even simpler and the details are omitted.
For all $(x,t)\in\rr_+^{n+1}$, let
$$
    F(x,t) \ev
\sum\limits_{j,k}|\lambda_{j,k}|\chi_{j,k}(x)\chi_{[2^{-(j+1)},2^{
-j}]}(t)\,.
$$
Discretizing the integral over $t$ by $t \sim 2^{-\ell}$, we
obtain
\begin{eqnarray}\label{5.15}
       &&\lf\|F|(\dpt)^{\sharp}\r\| \noz\\
       &&\hs\sim
\sup_{P\in\cq}\frac1{|P|^\tau}
\lf\|\chi_P\lf(\sum\limits_{\ell=j_P}^\fz 2^{\ell (s+n/q)q}
\int\limits_{2^{-(\ell+1)}}^{2^{-\ell}}
\lf[\sup\limits_{\gfz{y \in \rn}{t/2\le r\le t}}\frac{
|F(\cdot+y,r)|}{(1+|y|/r)^a}\r]^q\,\frac{dt}{t}\r)^{1/q}
\r\|_\lp\,.\quad\quad
\end{eqnarray}
With $t\in [2^{-(\ell+1)},2^{-\ell}]$ and $t/2\le r\le t$, we observe
that, for all $x\in\rn$,
\begin{equation}\label{5.16}
    F(x,r)= \sum\limits_{k\in\zz^n}\sum_{i=0}^1|\lambda_{\ell+i,k}|
    \chi_{\ell+i,  k}(x)\chi_{[2^{-(\ell+i+1)}, 2^{-(\ell+i)}]}(r)
\end{equation}
and hence
\begin{eqnarray}\label{5.17}
&&\lf\|F|(\dpt)^{\sharp}\r\|\noz\\
&&\hs\lesssim
\sup_{P\in\cq}\frac1{|P|^\tau}
\lf\|\chi_P\lf(\sum\limits_{\ell=j_P}^\fz2^{\ell (s+n/q)q}
\sup\limits_{y\in\R^n}\lf[\sum\limits_{k\in\zz^n}\frac{|\lambda_{\ell,k}|
\chi_{\ell,k}(\cdot+y)}{(1+2^{\ell}|y|)^a}\r]^q\,
\r)^{1/q}\r\|_\lp\,.\quad
\end{eqnarray}
In order to include also the situation $(p\wedge q) < 1$, we choose
$r>0$ so that $r<(1\wedge p\wedge q)$ and $ar>n$.
Obviously, we then estimate \eqref{5.17} in the following way
\begin{eqnarray}\label{5.18}
\lf\|F|(\dpt)^{\sharp}\r\|&&\lesssim
\sup_{P\in\cq}\frac1{|P|^\tau}\lf\|\chi_P\lf(\sum\limits_{\ell=j_P}^\fz\lf[
       \sum\limits_{k\in\zz^n}2^{\ell (s+n/q)r}|\lambda_{\ell,k}|^r\r.\r.\r.\noz\\
&&\hs\times\lf.\lf.\lf.\sup\limits_{y\in
\R^n}\frac{\chi_{\ell,k}(\cdot+y)}{(1+2^{\ell}|y|)^{ar}}\r]^{q/r}\,
       \r)^{1/q}\r\|_{\lp}\,.
\end{eqnarray}
We continue with the useful estimate \eqref{5.13}.
Since $ar>n$, the functions
$$g_{\ell}(\cdot)\ev 2^{\ell n}(1+2^{\ell}|\cdot|)^{-ar}$$
for all $\ell\in\zz$
belong to $L^1(\R)$ uniformly on $\ell$.
Applying Lemma \ref{l5.5} and putting \eqref{5.13} into
\eqref{5.18}, we obtain
\begin{eqnarray*}
       &&\lf\|F|(\dpt)^{\sharp}\r\|^r \\
       &&\hs\lesssim
\sup_{P\in\cq}\frac1{|P|^{r\tau}}  \lf\|\chi_P\lf\{\sum\limits_{\ell=j_P}^\fz
\lf[g_{\ell}\ast\lf(\sum\limits_{k \in \Z^n}2^{\ell (s+n/q)r }|
       \lambda_{\ell,k}|^r\chi_{\ell,k}\r)\r]^{q/r}\,
       \r\}^{r/q}\r\|_{L^{p/r}(\rn)}\,.
\end{eqnarray*}
If we use \eqref{eq:110322-1}, we then have
$$
 \lf\|F|(\dpt)^{\sharp}\r\| \lesssim
\sup_{P\in\cq}\frac1{|P|^{\tau}}
\lf\|\chi_P\lf(\sum\limits_{\ell=j_P}^\fz\sum\limits_{k \in
\Z^n}2^{\ell (s+n/q)q }|\lambda_{\ell,k}|^q\chi_{\ell,k}\,
   \r)^{1/q}\r\|_{\lp}.
$$

For the estimate from below, we go back to \eqref{5.15}
and observe that, for all $(x,t)\in\rr_+^{n+1}$,
$$
  \sup\limits_{y\in\rn}\frac{|F(x+y,t)|}{(1+2^{\ell}|y|)^a} \geq |F(x,t)|\,,
$$
which results in
\begin{equation}\nonumber
  \begin{split}
    \lf\|F|(\dpt)^{\sharp}\r\|
    &\gtrsim \sup_{P\in\cq}\frac1{|P|^\tau}\lf\|\chi_P
    \lf(\sum\limits_{\ell=j_P}^\fz 2^{\ell (s+n/q)
q}\int\limits_{2^{-(\ell+1)}}^{2^{-\ell}}|F(\cdot,t)|^q\,\frac{dt}{t}\r)^{
1/q}\r\|_{L^p(\R^n)}\,.\\
    &
  \end{split}
\end{equation}
A further application of \eqref{5.16} gives finally
that
$$
 \lf\|F|(\dpt)^{\sharp}\r\|
  \gtrsim \sup_{P\in\cq}\frac1{|P|^\tau}\lf\|\chi_P\lf(\sum\limits_{\ell=j_P}^\fz
  \sum\limits_{k \in
\Z^n}2^{\ell (s+n/q)q }|\lambda_{\ell,k}|^q\chi_{\ell,k}\,
   \r)^{1/q}\r\|_{L^p(\R^n)}\,,
$$
which completes the proof of Proposition \ref{prop:sawano-1}.
\end{proof}

\begin{proposition}\label{p5.4}
Let $s\in \R$ and $p\in (1,\infty)$.

{\rm (i)} If $q \in [1,\infty)$, $\tau \in [0,{1}/{(p\vee q)'}]$ and
$a \in (n({1}/{(p\wedge q)}+\tau),\infty)$, then
\begin{eqnarray*}
&&\lf\|\{\lambda_{j,k}\}_{j,k}|(\dlh)^{\sharp}\r\| \\
&&\hs\sim\inf_{\oz}\lf\{\sum\limits_{j\in\zz}2^{j(s+n/q)q}\lf[\int_\rn\lf(
\sum\limits_{k\in\zz^n}|\lambda_{j,k}|\chi_{j,k}(x) \r)^p[\oz(x,2^{-j})]^{-p}
\,dx\r]^{q/p}\r\}^{1/q}\,,
\end{eqnarray*}
where $\oz$ runs over all nonnegative measurable functions on $\rr^{n+1}_+$
satisfying \eqref{2.2}.

{\rm (ii)} If $q\in (1,\infty)$, $\tau \in
[0,{1}/{(p\vee q)'}]$ and $a\in
(n({1}/{(p\wedge q)}+\tau),\infty)$,
then
\begin{eqnarray*}
&&\lf\|\{\lambda_{j,k}\}_{j,k}|(\dph)^{\sharp}\r\|\\
&&\hs\sim
\inf_{\oz}\lf\{ \int_\rn \lf[\sum_{j\in\zz}2^{j(s+n/q)q}\lf(
\sum\limits_{k\in\zz^n}|\lambda_{j,k}|\chi_{j,k}(x) \r)^q[\oz(x,2^{-j})]^{-q}
\r]^{p/q}\,\,dx\r\}^{1/p}\,,
\end{eqnarray*}
where $\oz$ runs over all nonnegative measurable functions on $\rr^{n+1}_+$
satisfying \eqref{2.2}.
\end{proposition}

\begin{proof} The proof of this proposition
is similar to that of Proposition \ref{prop:sawano-1}
except that we use \cite[Lemma 7.13]{ysy}
instead of (\ref{eq:110322-1}), which completes the proof
of Proposition \ref{p5.4}.
\end{proof}

Propositions \ref{prop:sawano-1} and \ref{p5.4} justify the next
definition. Notice that we do not need the parameter $a\in(0,\fz)$ anymore on the
left-hand side.

\begin{definition} \label{def5.18} For all admissible parameters
{\rm (}in particular for $a\gg 1${\rm )},
let $\dot{f}^{s,\tau}_{p,q}
:= (\dot{P}^{s+n/2-n/q,\tau}_{p,q,a}(\cg))^{\sharp}$, $\dot{b}^{s,\tau}_{p,q}
:= (\dot{L}^{s+n/2-n/q,\tau}_{p,q,a}(\cg))^{\sharp}$, $b\dot{h}^{s,\tau}_{p,q}
:= (L\dot{H}^{s+n/2-n/q,\tau}_{p,q,a}(\cg))^{\sharp}$ and
$f\dot{h}^{s,\tau}_{p,q}
:= (P\dot{H}^{s+n/2-n/q,\tau}_{p,q,a}(\cg))^{\sharp}\,.$
\end{definition}

\section{Wavelet characterizations\label{s6}}

In this section, as an application of the coorbit
interpretations in Theorems \ref{t5.1} and
\ref{t5.2}, we obtain wavelet characterizations for the spaces
$\dot{B}^{s,\tau}_{p,q}(\R^n)$, $\dot{F}^{s,\tau}_{p,q}(\R^n)$,
$B\dot{H}^{s,\tau}_{p,q}(\R^n)$ and $F\dot{H}^{s,\tau}_{p,q}(\R^n)$ for the
full range of parameters extending the results in \cite[Section 8.2]{ysy}
where mainly
positive smoothness is considered. We use \emph{biorthogonal wavelet
bases} on $\R^n$ (see, for example, \cite{CoDaFe92,b03}), namely,
a system
\begin{equation}\label{6.1}
\langle \psi^0(\cdot-k), \widetilde{\psi}^0(\cdot-\ell)\rangle =
\delta_{k,\ell}\quad \mbox{and}\quad
\langle{\psi^1(2^{j}\cdot-k),\widetilde{\psi}^1(2^{\ell}\cdot-m)}\rangle =
\delta_{(j,k),(\ell,m)},
\end{equation}
with $k,\,\ell,\,m,\,j\in\zz$,
of scaling functions $(\psi^0, \widetilde{\psi}^0)$ and associated wavelets
$(\psi^1,\widetilde{\psi}^1)$, where $\delta_{k,\ell}=1$
if $k=\ell$ else $\delta_{k,\ell}=0$
and $\delta_{(j,k),(\ell,m)}=1$ if $j=\ell$ and $k=m$ else $\delta_{k,\ell}=0$.
Notice that the latter includes that, for every $f\in
L^2(\R)$,
$$
   f = \sum\limits_{j,k\in \Z} \langle f, \psi^1(2^j\cdot - k) \rangle
\widetilde{\psi}^1(2^j\cdot - k) = \sum\limits_{j,k\in \Z} \langle f,
\widetilde{\psi}^1(2^j\cdot - k) \rangle
\psi^1(2^j\cdot - k)
$$
holds in $L^2(\R)$.
We now construct a basis in $L^2(\R^n)$ by using the well-known tensor product
procedure. Consider the tensor products
$\Psi^c = \bigotimes_{j=1}^n \psi^{c_j}$ and $\widetilde{\Psi}^c =
\bigotimes_{j=1}^n \widetilde{\psi}^{c_j}$, $c\in E :=
\{0,1\}^n\setminus\{(0,...,0)\}$. The following result is well known for
orthonormal wavelets (see, for instance, \cite[Section 5.1]{Wo97}), however
straightforward to prove for biorthogonal wavelets.

\begin{lemma}\label{dwavelet}
For any given biorthogonal system
\eqref{6.1}, for all $f\in L^2(\R^n)$,
\begin{equation}\label{wavelet}
  f = \sum\limits_{c\in E}\sum\limits_{j\in \Z}\sum\limits_{k\in \Z^n} \langle
  f, 2^{\frac{jn}{2}}\Psi^{c}(2^j\cdot-k) \rangle
  2^{\frac{jn}{2}}\widetilde{\Psi}^{c}(2^j\cdot-k)
\end{equation}
with convergence in $L^2(\R^n)$\,.
\end{lemma}

Our aim is to
specify, namely, state precise sufficient conditions on $\psi^0$, $\psi^1$,
$\widetilde{\psi}^1$ and $\widetilde{\psi}^0$ such
that the representation \eqref{wavelet} extends to the spaces considered
above. Our method is to apply the abstract Theorem \ref{wbases} to this
specific situation. As the first step we check under which conditions the
functions $\Psi^c$, and $\widetilde{\Psi}^c$ belong to the set $B(Y)$, where $Y$
represents a Peetre-type space from Subsection \ref{s5.1}.

\begin{proposition}\label{propwiener} Let $L\in \N$, $K\in(0,\fz)$, and $\psi^0, \psi^1
\in L^2(\R)$, where $\psi^0$ satisfies $(D)$ and $(S_K)$, and
$\psi^1$ satisfies $(D)$, $(S_K)$ and $(M_{L-1})$\,. Let
$\cg$ be an $ax+b$-group and
$V \ev [-1,1]^n \times (1/2,1] \subset \mathcal{G}$ a neighborhood of the identity
$e\in \mathcal{G}$.
Suppose further that for $r_1,r_2\in \R$, the weight $w$ is given by
$w(x,t) \ev (1+|x|)^{v}(t^{r_2}+t^{-r_1})$ for all $(x,t) \in \mathcal{G}$\,.
If $p\in(0,1]$ and
\begin{equation}\label{6.3}
      L\wedge K>(r_1+{n}/{p}-{n}/{2})\vee (
r_2+v-{n}/{2}),
\end{equation}
then for all $c\in E$,
\begin{equation}\label{6.4}
    \int\limits_{\R^n}\int\limits_{0}^{\infty} \lf[\sup\limits_{(y,s)\in
(x,t)V}|\langle \pi(y,s)\Psi^c,\Psi^c\rangle|
    w(x,t)\r]^p \frac{dt}{t^{d+1}}\,dx<\infty\,.
\end{equation}
\end{proposition}

\begin{proof} By the assumptions on the wavelet system and $c \neq 0$, we see
that $\Psi^c$ satisfies $(D)$, $(S_K)$ and $(M_{L-1})$. Lemma \ref{help1}
implies that, for all $N \in \N$, $x\in\rn$ and $t\in(0,\fz)$,
   $$
       |(W_{\Psi^c} \Psi^c)(x,t)| \lesssim
\frac{t^{(L\wedge K)+n/2}}{(1+t)^{2(L\wedge K)+n}}\lf(1+\frac{|x|}{1+t}
\r)^{-N},
   $$
and hence
$$
   \sup\limits_{(y,s) \in (x,t)V} |W_{\Psi^c}\Psi^c(y,s)|
   =\sup\limits_{\substack{|y_i-x_i|\le t \\ t/2\le s\le t}}
|W_{\Psi^c}\Psi^c(y,s)| \lesssim
\frac{t^{(L\wedge K)+n/2}}{(1+t)^{2(L\wedge K)+n}}\lf(1+\frac{|x|}{1+t}
\r)^{-N}\,.
$$
Splitting the integral \eqref{6.4} into $\int_0^1 \int_{\R^n} + \int_1^{\infty}
\int_{\R^n}$ leads to the sufficient conditions \eqref{6.3} to ensure their
finiteness, which completes the proof of Proposition \ref{propwiener}.
\end{proof}

In what follows, we need the quantity
\begin{eqnarray}\label{6.5}
&&M(s,\tau,p,q,p^{\ast},a)\noz\\
&&\hs:=
\lf[s+n\lf(\tau+\frac 1{p^\ast}-\frac 1{p\vee 1}\r)+a\r]\vee
\lf[-s+n\lf(\tau+\frac 1{p^\ast}+\frac 1p-1\r)+2a\r].
\end{eqnarray}

From now on, $(\psi^0,\widetilde{\psi}^0)$ and $(\psi^1, \widetilde{\psi}^1)$ always
denote a biorthogonal wavelet system satisfying \eqref{6.1}. Our main results
read as follows.

\begin{theorem}\label{t6.1}
Let $s\in \R$, $\tau \in [0,\infty)$ and $q\in(0,\infty]$.

{\rm (i)} Let $p\in (0,\infty)$, $\psi^0$ and $\widetilde{\psi}^0$ satisfy $(D)$ and
$(S_K)$, and $\psi^1$ and $\widetilde{\psi}^1$ satisfy $(D)$, $(S_K)$ and $(M_{L-1})$,
where
$(L\wedge K)>M\lf(s,\tau,p,q, 1\wedge p\wedge q,{n}/{
(p\wedge q)}\r)\,$
in view of \eqref{6.5}, then every $f \in \dot{F}^{s,\tau}_{p,q}(\R^n)$
admits the decomposition
\eqref{wavelet}, where the sequence
\begin{equation}\label{6.7}
\lambda^c \ev
\{\lambda^c_{j,k}\}_{j\in \Z,k\in Z^n}\ \text{with}\ \lambda^c_{j,k}:= \langle
  f, 2^{\frac{jn}{2}}\Psi^{c}(2^j\cdot-k) \rangle,\quad j\in \Z,\ k\in
\Z^n\,,
\end{equation}
belongs to the sequence space $\dot{f}^{s,\tau}_{p,q}$.

Conversely, an element
$f\in (\mathcal{H}^1_{w_Y})^{\sim}$, where $Y \ev
\dot{P}^{s+n/2-n/q,\tau}_{p,q,a}(\cg)$, belongs to $\dot{F}^{s,\tau}_{p,q}(\R^n)$ if
the sequence \eqref{6.7} belongs to $\dot{f}^{s,\tau}_{p,q}$.

{\rm (ii)} If $p \in (0,\infty]$, $\psi^0$ and $\widetilde{\psi}^0$ satisfy $(D)$ and
$(S_K)$, and $\psi^1$ and $\widetilde{\psi}^1$ satisfy $(D)$, $(S_K)$ and $(M_{L-1})$,
where
$(L\wedge K)>M\lf(s,\tau,p,q, 1\wedge p\wedge q, {n}/{p}\r)\,$
in view of \eqref{6.5}, then every $f \in
\dot{B}^{s,\tau}_{p,q}(\R^n)$ admits the decomposition
\eqref{wavelet}, where the sequence \eqref{6.7} belongs to the
sequence space $\dot{b}^{s,\tau}_{p,q}$.

Conversely, an element
$f\in (\mathcal{H}^1_{w_Y})^{\sim}$, where $Y \ev
\dot{L}^{s+n/2-n/q,\tau}_{p,q,a}(\cg)$, belongs to
$\dot{B}^{s,\tau}_{p,q}(\R^n)$ if the sequence \eqref{6.7}
belongs to $\dot{b}^{s,\tau}_{p,q}$.
\end{theorem}

\begin{theorem}\label{t6.2} Let $s\in \R$, $p \in (1,\infty)$, $q \in
[1,\infty)$, and $\tau \in [0,{1}/{(p\vee q)'}]$.

{\rm (i)} Let $\psi^0$ and
$\widetilde{\psi}^0$ satisfy $(D)$ and $(S_K)$, and $\psi^1$ and $\widetilde{\psi}^1$ satisfy
$(D)$, $(S_K)$ and $(M_{L-1})$,
where
$(L\wedge K)>M\lf(s,\tau,p,q,{1}/{((p\vee q)'+1)}, n\lf({1}/{
p}+\tau\r) \r)\,$
in view of \eqref{6.5}. Then every $f \in
B\dot{H}^{s,\tau}_{p,q}(\R^n)$ admits the decomposition
\eqref{wavelet}, where the sequence \eqref{6.7} belongs to the
sequence space $b\dot{h}^{s,\tau}_{p,q}$.

Conversely, an element
$f\!\in\! (\mathcal{H}^1_{w_Y})^{\sim}$, where
$Y\!\ev\!L\dot{H}^{s+n/2-n/q,\tau}_{p,q,a}(\cg),$
belongs to
$B\dot{H}^{s,\tau}_{p,q}(\R^n)$ if the sequence \eqref{6.7}
belongs to $b\dot{h}^{s,\tau}_{p,q}$.

{\rm (ii)} Let $q\in(1,\fz)$, $\psi^0$ and
$\widetilde{\psi}^0$ satisfy $(D)$ and $(S_K)$, and $\psi^1$ and $\widetilde{\psi}^1$ satisfy
$(D)$, $(S_K)$ and $(M_{L-1})$,
where
$(L\wedge K )>M\lf(s,\tau,p,q,{1}/{((p\vee q)'+1)},n\lf({1}/{
(p\wedge q)}+\tau\r) \r)\,$
in view of \eqref{6.5}.
Then every $f \in F\dot{H}^{s,\tau}_{p,q}(\R^n)$ admits the decomposition
\eqref{wavelet}, where the sequence \eqref{6.7}
belongs to the sequence space $f\dot{h}^{s,\tau}_{p,q}$.

Conversely, an element
$f\!\in\!(\mathcal{H}^1_{w_Y})^{\sim}$, where
$Y\!\ev\!P\dot{H}^{s+n/2-n/q,\tau}_{p,q,a}(\cg),$
belongs to $F\dot{H}^{s,\tau}_{p,q}(\R^n)$
if the sequence \eqref{6.7} belongs to $f\dot{h}^{s,\tau}_{p,q}$.
\end{theorem}

\begin{proof} We prove both theorems simultaneously and start with the
interpretations as coorbits in Theorems \ref{t5.1} and \ref{t5.2}. Using Propositions
\ref{p5.1} and \ref{p5.2}, together with \eqref{wy}, we conclude that, when
$Y$ is one of the spaces $\dot{P}^{s+n/2-n/q,\tau}_{p,q,a}(\cg),\ \dot{L}^{s+n/2-n/q,\tau}_{p,q,a}(\cg),\
L\dot{H}^{s+n/2-n/q,\tau}_{p,q,a}(\cg)$ and $P\dot{H}^{s+n/2-n/q,\tau}_{p,q,a}(\cg),$
the relation
$w_{Y}(x,t) \le (1+|x|)^v(t^{r_2}+t^{-r_1})$
holds true for all $x\in\rn$ and $t\in(0,\fz)$, where
$$r_1 = \max\{s+n\tau+n(1/2-1/p),
-s+n\tau+n(1/p-1/2),s+a-n/2,-s-n/2+a\}\,,$$
$$r_2= \max\{-s+n\tau+n(1/p-1/2),s-n(1/p-1/2)+n\tau,s+n/2,-s+n/2+a\}\,,$$
and $v= a$.
We then use the abstract result in Theorem \ref{wbases}, together with
\eqref{wavelet}, \eqref{4.4} and Proposition \ref{propwiener}. Notice that the
Aoki-Rolewicz-exponent $p^{\ast} = {1}/{((p\vee q)'+1)}$ in
Theorem \ref{t6.2} and
$p^{\ast} = (1\wedge p\wedge q)$ in Theorem \ref{t6.1},
which is needed in
\eqref{4.4} and Proposition \ref{propwiener}, is
given in Remark \ref{r3.2}(ii) and Remark \ref{r5.1}(ii). Taking in addition
the conditions on $a\in(0,\fz)$ in Theorems \ref{t5.1} and \ref{t5.2} into account,
we conclude the proofs of Theorems \ref{t6.1} and \ref{t6.2}.
\end{proof}

\begin{remark}\label{r6.-1}
The wavelet decomposition characterization in terms of the Carleson measures
for $\bmo$ is already obtained (see \cite[p.\,154, Theorem 4]{M92}),
namely, $\bmo$ is characterized
by wavelets in terms of the space
$\dot{f}_{p,2}^{0,1/p}(\rn)$. The conditions on the wavelets used therein are much
weaker than ours; see Theorem \ref{t6.1}(i). Indeed, using the equivalence
$\mathop\mathrm{BMO}({\mathbb R}^n) = F^{0,1/p}_{p,2}({\mathbb R^n})$, where $p$
is close to infinity, we obtain the condition $K,L\in(n,\fz)$ in Theorem \ref{t6.1}.
Thus, we need more smoothness and vanishing moments for the
wavelets. However, for spaces on the real line (namely, $n=1$),
the condition is close to the optimal one in \cite{M92}.
\end{remark}

\begin{remark}\label{r6.0}
Comparing with the \emph{orthogonal wavelet} characterizations of $\dbt$ and
$\dft$ obtained in \cite[Theorem 8.2]{ysy}, therein the parameter $s$
restricts to $s\in(0,\infty)$, we establish the
\emph{biorthogonal wavelet} characterizations of these spaces for all
$s\in\rr$ in Theorem \ref{t6.1}. The wavelet
characterizations for Besov-Hausdorff spaces $\dbh$ and
Triebel-Lizorkin-Hausdorff spaces $\dfh$ in Theorem \ref{t6.2} are
totally new even for some special cases such as Hardy-Hausdorff
spaces.
\end{remark}

\begin{remark}\label{r6.1} Notice that we use more general \emph{biorthogonal
wavelet bases} instead of classical orthogonal wavelet bases as it is done for
instance in \cite[Chapter 3]{t06}, in case of inhomogeneous
classical Besov and Triebel-Lizorkin spaces, and also in
\cite{fjw,HaPi08,ht05, htr05,b05,bh, RaUl10}.
We refer to \cite{cy} for the spaces $Q_\az(\rn)$,
and to \cite{s08} for the Triebel-Lizorkin-Morrey spaces. For some
earlier results with biorthogonal wavelets on classical Besov-Triebel-Lizorkin
spaces, we refer to \cite{Ky03}. Notice also that the conditions on $K$ and $L$
in Theorems \ref{t6.1} and \ref{t6.2} are not optimal. However,
they represent precise sufficient conditions.
\end{remark}

We conclude the paper with the following corollary, which
characterizes the set
of parameters, when the used wavelet system represents an unconditional basis.
Notice that by the abstract Theorem \ref{wbases}, we always have the unconditional
convergence of \eqref{wavelet} in the weak $\ast$-topology induced by
$(\mathcal{H}^1_{w_Y})^{\sim}$. Since the coefficients in the expansion
\eqref{wavelet} are always unique due to the biorthogonality relations
\eqref{6.1}, this can be seen as a sort of the weak unconditional basis. For an
unconditional basis (in the strong sense) we further need the
unconditional convergence of \eqref{wavelet} in the (quasi-)norm of the
respective space.

\begin{corollary}\label{c6.8}
Under the assumptions of Theorems \ref{t6.1} and
\ref{t6.2}, the used wavelet systems represent an unconditional basis in the
respective spaces if and only if $\tau = 0$ and $p,q\in(0,\infty)$.
\end{corollary}

\begin{proof}  When $\tau=0$ and $p,q\in(0,\infty)$,
this follows from the fact
that the finite sequences form a dense subspace in the respective sequence
space. Let us disprove the unconditional basis property for the space
$\dot{F}^{s,\tau}_{p,q}(\mathbb{R}^n)$. For the other cases, the proof is
similar. If $\tau\in(0,\fz)$, we use the distribution
\begin{equation}\label{f2}
f(\cdot):=\sum_{j=1}^\infty
2^{-j(n\tau+s-n/p)}\widetilde{\Psi}^c(2^j\cdot-(2,2,\cdots,2))
\end{equation}
for some $c\in E$.
Indeed, the right-hand side converges unconditionally in the topology of
$(\mathcal{H}^1_{w_Y})^{\sim}$, where $Y \ev
\dot{P}^{s+n/2-n/q,\tau}_{p,q,a}(\cg)$, and defines an element from the
reservoir space $(\mathcal{H}^1_{w_Y})^{\sim}$. By the
biorthogonality relations \eqref{6.1}, together with Proposition
\ref{prop:sawano-1} and Definition \ref{def5.18}, we see that the sequence
\eqref{6.7} belongs to the respective sequence space $\dot{f}^{s,\tau}_{p,q}(\rn)$.
Indeed, since $\tau\in(0,\fz)$,
we have
\begin{align*}
\|\{\lambda^c_{j,k}\}_{j,k}|\dot{f}^{s,\tau}_{p,q}\|
&\sim
\sup_{Q \in \cq}
\frac{1}{|Q|^\tau}
\left\{
\int_Q
\sum_{\ell = j_Q}^\infty 2^{-\ell p n\tau+\ell n}\chi_{\ell,(2,2,\cdots,2)}(x)\,dx
\right\}^{1/p}\\
&\le
\sup_{\N\in \Z}
\frac{1}{2^{-Nn\tau}}
\left\{
\sum_{\ell = N}^\infty 2^{-\ell pn\tau}
\right\}^{1/p}\le \sup_{\N\in \Z}
2^{Nn\tau}2^{-Nn\tau} \sim 1.
\end{align*}
Hence, by Theorem \ref{t6.1}, $f$ belongs to the space
$\dot{F}^{s,\tau}_{p,q}(\mathbb{R}^n)$ and \eqref{f2} is its biorthogonal
wavelet expansion. By using the equivalent sequence space (quasi-)norm
in Proposition \ref{prop:sawano-1}(ii), we see, by a similar computation, that
\eqref{f2} does not converge in the (quasi-)norm of
$\dot{F}^{s,\tau}_{p,q}(\mathbb{R}^n)$, which completes the proof of Corollary \ref{c6.8}.
\end{proof}

\noindent{\bf Acknowledgements}\quad The authors would like to thank
the anonymous referees and Hans-Georg Feichtinger for making several useful
suggestions which helped to improve the paper and its
readability.

\bigskip

\noindent Yiyu Liang, Dachun Yang (Corresponding author) and Wen Yuan

\medskip

\noindent School of Mathematical Sciences, Beijing Normal University,
Laboratory of Mathematics and Complex Systems, Ministry of
Education, Beijing 100875, People's Republic of China

\smallskip

\noindent {\it E-mails}: \texttt{yyliang@mail.bnu.edu.cn} (Y. Liang)

\hspace{1.13cm}\texttt{dcyang@bnu.edu.cn} (D. Yang)

\hspace{1.13cm}\texttt{wenyuan@bnu.edu.cn} (W. Yuan)

\bigskip

\noindent Yoshihiro Sawano

\medskip

\noindent Department of Mathematics and Information Sciences,
Tokyo Metropolitan University, Minami-Ohsawa 1-1, Hachioji-shi,
Tokyo 192-0397, Japan.

\smallskip

\noindent {\it E-mail}: \texttt{yoshihiro-sawano@celery.ocn.ne.jp}

\bigskip

\noindent Tino Ullrich

\medskip

\noindent  Hausdorff Center for Mathematics \& Institute for Numerical Simulation,
Endenicher Allee 60, 53115 Bonn, Germany

\smallskip

\noindent {\it E-mail}: \texttt{tino.ullrich@hcm.uni-bonn.de}

\end{document}